\documentclass[11pt, leqno]{amsart}
\usepackage{amsmath, amssymb, latexsym}
\usepackage{mathrsfs}
\usepackage{enumerate}
\usepackage{color}
\usepackage[colorlinks=true, pdfstartview=FitV, linkcolor=blue,%
citecolor=blue, urlcolor=blue]{hyperref}
\usepackage[all, knot]{xy}

\numberwithin{equation}{section}

\newtheorem{theorem}{Theorem}[section]
\newtheorem{corollary}[theorem]{Corollary}
\newtheorem{lemma}[theorem]{Lemma}
\newtheorem{proposition}[theorem]{Proposition}

\theoremstyle{definition}
\newtheorem{definition}[theorem]{Definition}
\newtheorem{remark}[theorem]{Remark}

\newcommand{\Hom}{\operatorname{Hom}}

\newcommand{\C}{\mathbb{C}}

\newcommand{\Q}{\mathbb{Q}}
\newcommand{\N}{\mathbb{N}}
\newcommand{\Z}{\mathbb{Z}}
\newcommand{\h}{\mathfrak{h}}
\newcommand{\g}{\mathfrak{g}}
\newcommand{\R}{\mathscr{R}}
\newcommand{\U}{U_{q}(\mathfrak{g})}

\title[quantum Borcherds-Bozec superalgebras]
{quantum Borcherds-Bozec superalgebras}

\author[Zhaobing Fan]{Zhaobing Fan}
\address{Harbin Engineering University,
	Harbin, China}
\email{fanzhaobing@hrbeu.edu.cn}

\author[Jiaqi Huang]{Jiaqi Huang}
\address{Harbin Engineering University,
	Harbin, China}
\email{jiaqihuang@hrbeu.edu.cn}
\thanks{}

\keywords{quantum Borcherds-Bozec superalgebra, highest weight module, character formula}

\subjclass[2010] {17B37, 17B67, 16G20}

\begin{document}

\begin{abstract}
 We introduce quantum Borcherds-Bozec superalgebras. We present and prove various results of the quantum superalgebras including a bilinear form, higher Serre relation, quasi-R-matrix,
 character formula for the irreducible highest weight modules. We also prove the category of integrable
 representations is semi-simple.
\end{abstract}

\maketitle

\section*{Introduction}

In \cite{Bor88}, R. E. Borcherds introduced {\it Borcherds algebras} associated with Borcherds-Cartan matrices.
The diagonal entries of these matrices may be non-positive, which means the Borcherds algebras
may have simple roots with the norm $\le 0$, the {\it imaginary} simple roots.
Hence the index set $I$ has a partition $I = I^{\text{re}} \sqcup I^{\text{im}}$,
where  $I^{\text{re}} = \{i \in I \mid a_{ii} = 2\}$,
$I^{\text{im}} = \{i \in I \mid a_{ii} \le 0\}$.
Borcherds algebras
played an important role in Borcherds' proof of the   famous {\it  Moonshine conjecture} \cite{Bor92}.
\vskip 2mm
 T.Bozec  introduced the {\it quantum Borcherds-Bozec algebras} in his geometric investigation of the representation theory of quivers with loops \cite{Bozec2014a,Bozec2014b}. By using simple perverse sheaves on the representation varieties of such quivers, he provided a construction of Lusztig＊s canonical basis for the positive half of a quantum Borcherds-Bozec algebra.
The Cartan matrices for quantum Borcherds-Bozec algebras are the same for quantum Borcherds algebras, but they have higher degree simple root vectors.
The index set for these simple root vectors is denoted by
$I^{\infty}=(I^{\text{re}}\times \{1\}) \sqcup (I^{\text{im}}\times \Z_{\textgreater 0})$.
The Weyl-Kac type character formula for the irreducible highest weight modules was proved in \cite{BSV2016}.
In \cite{FKKT21,FKRT21,FHKX}, Fan and his collaborators have studied the global bases and crystal bases of quantum Borcherds-Bozec algebras.
 Kang, Kim and Tong \cite{KKT} categorified the quantum Borcherds-Bozec algebra  for an arbitrary
 Borcherds-Cartan datum by constructing their associated Khovanov-Lauda-Rouquier algebras and they showed that the cyclotomic Khovanov-Lauda-Rouquier
 algebras provide a categorification of the irreducible highest weight modules.
\vskip 2mm
On the other hand, quantum superalgebras have developed rapidly, particularly with remarkable breakthroughs in structural analysis, representation theory and cross-disciplinary applications.
 In \cite{BKM}, Benkart-Kang-Melville gave the definition of quantum Borcherds superalgebras. In \cite{CHW},  S. Clark, D. Hill, and W. Wang introduced a quantum
 covering and super groups with no isotropic odd simple root and provided substantial research achievements.
In \cite{CHW2}, S. Clark, Z. Fan, Y. Li and W. Wang showed that the modified
 quantum group and supergroup are isomorphic over the rational function field adjoined
 with $\sqrt{-1}$, by constructing a twistor on the modified covering quantum group. We proved the
Weyl-Kac type character formula for the irreducible highest
weight modules with dominant integral highest weights of Borcherds-Bozec superalgebras\cite{FHKS}.
 \vskip 2mm
The goal of this paper is to lay the foundations  for the theory of quantum Borcherds-Bozec superalgebras . We define quantum Borcherds-Bozec superalgebras.
We prove various
 structural results of the quantum superalgebras including a bilinear form,integral forms, bar-involution, quantum higher Serre relation, quasi-R-matrix, Casimir, characters for
 integrable modules, we also investigate representation theory, we prove that integrable modules of this quantum superalgebras are completely reducible.
\vskip 2mm
Quantum Borcherds-Bozec superalgebras are a generalization of the quantum
supergroups proposed by Wang and his collaborators in a series of papers. Compared with the theory of quantum supergroups with no isotropic odd simple root,
 the difficulties in quantum Borcherds-Bozec superalgebras come from the countable infinite
 index set, the imaginary roots and the superalgebra structure. These cause the higher Serre relations in  our algebras to be more complicated and when we aim to show that the highest weight module in the category $O_{int}$ is simple, we can not follow the methods provided by Lusztig or Kang, instead we present a new approach.
When Bozec introduced quantum Borcherds-Bozec algebras, he provided Chevalley generators and primitive generators. Unlike  Chevalley generators, primitive generators have a more concise coproduct. In this paper, we adopt primitive generators to facilitate computations and simplify the coproduct and generating relations.

\vskip 2mm
This paper is organized as follows. In Section 1, We give an description of the radical of Lusztig's bilinear form.  In Section 2, we introduce the quantum Borcherds-Bozec superalgebra and investigate the property of the highest weight modules. In Section 3, we prove that the quantum Borcherds-Bozec superalgebras have a triangular
 decomposition.  In Section 4, we formulate the quasi-R-matrix and
 establish its basic properties.  We construct the quantum Casimir. Finally, we prove that the category $O_{int}$ is semisimple and give the character formula for the simple objects of $O_{int}$.
 \vskip 2mm
 This paper lays the foundation for further studies of quantum Borcherds-Bozec superalgebras. In our next  paper, we will construct  the crystal basis of quantum Borcherds-Bozec superalgebras and of integral modules.
\vskip 2mm
{\it Acknowledgements.} Z.Fan was partially supported by the NSF of China grant $12271120$.
 We would like to express our sincere gratitude to Professor Seok-Jin
 Kang for his helpful discussions.
\vskip 5mm

\section{The Algebra F}

Let $I$ be an index set which can be countably infinite.
An integer-valued matrix $A=(a_{ij})_{i,j \in I}$ is called  a {\it
symmetrizable Borcherds-Cartan matrix} if it satisfies the following conditions:
\begin{itemize}
\item[(i)] $a_{ii}=2, 0, -2, -4, ...$,

\item[(ii)] $a_{ij}\le 0$ for $i \neq j$,

\item[(iii)] there exists a diagonal matrix $D=\text{diag} (d_{i} \in \Z_{>0} \mid i \in I)$ such that $DA$ is symmetric.
\end{itemize}

\vskip 2mm

Set $I^{\text{re}}=\{i \in I \mid a_{ii}=2 \}$,
$I^{\text{im}}=\{i \in I \mid a_{ii} \le 0\}$ ,
$I^{\text{iso}}=\{i \in I \mid a_{ii}=0 \}$ and $I^{\infty}:= (I^{\text{re}} \times \{1\}) \sqcup (I^{\text{im}}
\times \Z_{>0})$. For simplicity, we will often write $i$ for $(i,1)$.

\vskip 2mm

{ An {\it super Borcherds-Cartan datum}} consists of :

\begin{itemize}

\item[(a)] a symmetrizable Borcherds-Cartan matrix $A=(a_{ij})_{i,j \in I}$,

\item[(b)] a free abelian group $P$, the {\it weight lattice},

\item[(c)] $P^{\vee} := \Hom(P, \Z)$, the {\it dual weight lattice},

\item[(d)] $\Pi=\{\alpha_{i} \in P  \mid i \in I \}$, the set of {\it simple roots},

\item[(e)] $\Pi^{\vee}=\{h_i \in P^{\vee} \mid i \in I \}$, the set of {\it simple coroots},

\item[(f)]  $I=I_{\bar{0}}\sqcup I_{\bar{1}}$,

\end{itemize}

\noindent satisfying the following conditions

\begin{itemize}

\item[(i)] $\langle h_i, \alpha_j \rangle = a_{ij}$ for all $i, j \in I$,

\item[(ii)] $\Pi$ is linearly independent over $\C$,

\item[(iii)] { for all $i \in I_{\bar{1}}$, we have $a_{ij} \in 2 \Z $},

\item[(iv)] for each $i \in I$, there exists an element $\Lambda_{i} \in P$ such that
$$\langle h_j , \Lambda_i \rangle = \delta_{ij} \ \ \ \text{for all} \ i, j \in I,$$

\end{itemize}

\vskip 2mm

{ We denote $I_{\bar{0}} = I \setminus I_{\bar{1}}$, the set of {\it even indices}}.
Given  a symmetrizable Borcherds-Cartan matrix, such
an super Borcherds-Cartan datum always exists, which
is not necessarily unique. The $\Lambda_i$  $(i \in I)$ are called the {\it fundamental weights}.

An super Borcherds-Cartan datum is called bar-consistent if it satisfies
$$d_{i}\equiv p(i)\pmod{2}, \ \forall i\in I,$$
where $p$ is a parity function $p:I\rightarrow \{0,1\}$ such that $p(I_{\bar{1}})=1,p(I_{\bar{0}})=0$.
\vskip 2mm

We define
\begin{equation} \label{eq:P}
P^{+}:=\{\lambda \in P \mid \langle h_i, \lambda \rangle \ge 0 \ \text{for} \ i \in I,
\ \ \langle h_{i}, \lambda \rangle \in 2 \Z \ \text{for} \ i \in I^{\text{re}} \cap I_{\bar{1}}  \}.
\end{equation}
The elements in $P^{+}$ are called the {\it dominant integral weights}.

\vskip 2mm

The free abelian group $Q:= \bigoplus_{i \in I} \Z \, \alpha_i$ is called the {\it root lattice}.
Set $Q_{+}: = \sum_{i \in I} \Z_{\ge 0}\, \alpha_{i}$ and $Q_{-}: = -Q_{+}$.
For $\beta = \sum k_i \alpha_i \in Q\textbf{}_{+}$, we define its {\it height} to be
{ $\text{ht} (\beta):=\sum k_i$}.

\vskip 2mm

Let ${\mathfrak h} := \C \otimes_{\Z} P^{\vee}$ be the {\it Cartan subalgebra}.
We define
a partial order on ${\mathfrak h}^{*}$ by setting $\lambda \ge \mu$
if and only if $\lambda - \mu  \in Q_{+}$ for $\lambda, \mu \in {\mathfrak h}^{*}$.

\vskip 2mm

 We decompose
${\mathfrak h}= \bigoplus_{i \in I} \C \, h_{i}$ .
Since $A$ is symmetrizable, we can define
a symmetric bilinear form $( \ \, , \ \, )$ on ${\mathfrak h}$ by
\begin{equation} \label{eq:h-bilinear}
\begin{aligned}
& (h, h_i) = d_{i}^{-1} \langle h, \alpha_ {i} \rangle \ \ \text{for} \ h \in {\mathfrak h}, \\
& (h', h'') = 0 \  \ \text{for} \ h', h'' \in {\mathfrak h}''.
\end{aligned}
\end{equation}

{In particular}, we have
$$(h_i, h_j) = d_{j}^{-1} a_{ij} = d_{i}^{-1} a_{ji} = (h_j, h_i).$$

\vskip 2mm

Moreover, since $\Pi$ is linearly independent,
it is straightforward to verify that $( \ \, , \ \, )$ is non-degenerate on ${\mathfrak h}$.
Hence there is a linear isomorphism $\nu: \mathfrak{h} \rightarrow \mathfrak{h}^{*}$
defined by
\begin{equation} \label{eq:linear_iso}
\nu(h)(h') = (h, h') \ \ \text{for} \  h, h' \in {\mathfrak h},
\end{equation}
{ which implies}
$$\nu(h_i) = d_{i}^{-1} \alpha_{i}, \ \ \nu^{-1}(\alpha_i) = d_i h_i.$$

Then we obtain the induced non-degenerate symmetric bilinear for $(\ \, , \ \, )$
on ${\mathfrak h}^{*}$ defined by
\begin{equation} \label{eq:hstar}
(\lambda, \mu): = (\nu^{-1}(\lambda), \nu^{-1}(\mu)) \ \ \text{for} \ \lambda, \mu \in {\mathfrak h}^{*}.
\end{equation}
In particular, we have
\begin{equation} \label{eq:alpha_i}
(\alpha_i, \lambda) = d_{i} \langle h_{i}, \lambda \rangle \ \ \text{for} \ i \in I.
\end{equation}

\vskip 2mm

For each  $i \in I^{\text{re}}$, we  define the {\it simple reflection}
$r_{i}:{\mathfrak h}^{*} \rightarrow {\mathfrak h}^{*}$ by
\begin{equation} \label{eq:simple_reflection}
r_{i}(\lambda)= \lambda - \langle h_{i}, \lambda \rangle \,  \alpha_{i}
\ \ \text{for} \ \lambda \in {\mathfrak h}^{*}.
\end{equation}

The subgroup $W$ of $GL({\mathfrak h}^{*})$ generated by the simple reflections $r_{i}$ $(i \in I^{\text{re}})$ is called the {\it Weyl group} of the super Borcherds-Cartan datum given above.
Similar to \cite{FKRT21} , \  $(\ \,  , \ \,  )$ is $W$-invariant.

\vskip 2mm

For the rest of the paper, we fix a linear functional $\rho \in {\mathfrak h}^{*}$ satisfying
\begin{equation} \label{eq:rho}
(\rho, \alpha_{i}) = \ \dfrac{1}{2} (\alpha_{i}, \alpha_{i}) \ \ \text{for all} \ i \in I.
\end{equation}

For any $i\in I$, we set
$$q_{i}=q^{d_{i}},\ q_{(i)}=q^{\frac{(\alpha_{i},\alpha_{i})}{2}}$$
\vskip 3mm

Let \(F\) be the free associative \(\mathbb{Q}(q)\)-superalgebra generated by the unit \(1\) and elements \(a_{il}\), where \((i, l) \in I^{\infty}\). We endow \(F\) with a \emph{parity grading} (i.e., a \(\mathbb{Z}_2\)-grading) defined on generators by \(p(a_{il}) = l\,p(i) \pmod{2}\) and a \emph{weight grading} by the root lattice, given by \(\operatorname{wt}(a_{il}) = l\alpha_i\). For a homogeneous element \(x \in F\), we denote its weight by \(|x|\).

The tensor product \(F \otimes F\) is made into a \(\Q(q)\)-superalgebra via the multiplication
\[
(x_1 \otimes x_2)(x_3 \otimes x_4) = (-1)^{p(x_2)p(x_3)} q^{(|x_2|,|x_3|)} \, x_1x_3 \otimes x_2x_4.
\]

Throughout this paper, we assume that all elements appearing in displayed formulas are homogeneous with respect to the \(\mathbb{N}[I] \times \mathbb{Z}_2\)-bigrading.

\vskip 3mm
There is similar multiplication formula in $F\otimes F\otimes F$

\begin{align*}
&(x_{1}\otimes x_{2}\otimes x_{3})(x_{1}'\otimes x_{2}'\otimes x_{3}')\\
=&q^{(|x_2|,|x_1'|)+(|x_3|,|x_2'|)+(|x_3|,|x_1'|)}(-1)^{p(x_2)p(x_1') + p(x_3)p(x_2')+p(x_3)p(x_1')}x_1x_1' \otimes x_2x_2' \otimes x_3x_3'.
\end{align*}
\vskip 3mm
We endow $F$ with a map $\varrho: F\to F\otimes F$ defined by
$$\varrho(a_{il})=a_{il}\otimes 1+1\otimes a_{il}.$$

\vskip 3mm
\begin{proposition}\label{1.1}
There exists a unique bilinear form $(\cdot,\cdot): F \times F \to \Q(q)$ such that $(1,1)=1$ and the following properties hold for all homogeneous elements:
\begin{itemize}
    \item[(a)] $\displaystyle (a_{il},a_{jk}) = \delta_{ij}\delta_{lk} \bigl(1-(-1)^{p(li)}q_{i}^{2l}\bigr)^{-1}$, \quad $\forall(i,l),(j,k) \in I^{\infty}$;
    \item[(b)] $\displaystyle (x, yy') = \bigl(\varrho(x),\, y \otimes y'\bigr)$, \quad $\forall x,y,y' \in F$;
    \item[(c)] $\displaystyle (xx', y) = \bigl(x \otimes x',\, \varrho(y)\bigr)$, \quad $\forall x,x',y \in F$.
\end{itemize}
The same symbol $(\cdot,\cdot)$ will be used to denote the induced bilinear form on $F \otimes F$, defined by
\[
(x_1 \otimes x_2,\; x_3 \otimes x_4) = (x_1,x_3)(x_2,x_4).
\]
Moreover, the bilinear form on $F$ is symmetric.
\end{proposition}

\begin{proof}
We equip the graded dual $F^* = \bigoplus_\nu F_\nu^*$ with an associative algebra structure by transposing the coproduct $\varrho: F \to F \otimes F$. Concretely, for $f,g\in F^*$ we set
\[
(fg)(x) = (f \otimes g)(\varrho(x)), \qquad \forall x \in F.
\]

For each pair $(i,l) \in I^{\infty}$, let $\xi_{il}\in F^*$ be the linear map determined by
\[
\xi_{il}(a_{jk}) = \delta_{ij}\delta_{lk}\bigl(1-(-1)^{p(li)}q_i^{2l}\bigr)^{-1}.
\]
Let $\phi: F \to F^*$ be the unique algebra homomorphism satisfying $\phi(a_{il}) = \xi_{il}$ for all $(i,l)\in I^{\infty}$. The map $\phi$ preserves the $\mathbb{N}[I]\times\mathbb{Z}_2$-grading.

Define a bilinear form on $F$ by
\[
(x,y) = \phi(y)(x), \qquad \forall x,y\in F.
\]
 (a) follows immediately from the definition of $\xi_{il}$. Because $\phi$ is an algebra homomorphism, we have for any $x,y,y'\in F$
\[
(x, yy') = \phi(yy')(x) = (\phi(y)\phi(y'))(x) = (\phi(y)\otimes \phi(y'))(\varrho(x)) = (\varrho(x),\, y\otimes y'),
\]
which is exactly  (b). Clearly, the bilinear form satisfies
\begin{itemize}
    \item[(d)] $(x,y)=0$ whenever $x$ and $y$ are homogeneous with $|x|\neq |y|$.
\end{itemize}

It remains to show (c) holds. We prove it by induction on the weight of $y$.  Assume that (c) is true for $y$ and $y'$ and for all $x,x'$. We  prove that (c) holds for t $y'' = yy'$ and any $x,x'$.
Write
\[
\varrho(x) = \sum x_1\otimes x_2, \quad \varrho(x') = \sum x_1'\otimes x_2', \quad
\varrho(y) = \sum y_1\otimes y_2, \quad \varrho(y') = \sum y_1'\otimes y_2'.
\]
Then
\begin{align*}
\varrho(xx') &= \sum q^{(|x_2|,|x_1'|)} (-1)^{p(x_2)p(x_1')} \, x_1x_1' \otimes x_2x_2', \\
\varrho(yy') &= \sum q^{(|y_2|,|y_1'|)} (-1)^{p(y_2)p(y_1')} \, y_1y_1' \otimes y_2y_2'.
\end{align*}
By the definition and the induction hypothesis, we have
\begin{align*}
(xx', yy') &= (\phi(y)\phi(y'))(xx')
            = (\phi(y)\otimes \phi(y'))(\varrho(xx')) \\
           &= \sum q^{(|x_2|,|x_1'|)} (-1)^{p(x_2)p(x_1')}
              (x_1x_1', y) (x_2x_2', y') \\
           &= \sum q^{(|x_2|,|x_1'|)} (-1)^{p(x_2)p(x_1')}
              \bigl(x_1\otimes x_1', \varrho(y)\bigr)
              \bigl(x_2\otimes x_2', \varrho(y')\bigr) \\
           &= \sum q^{(|x_2|,|x_1'|)} (-1)^{p(x_2)p(x_1')}
              (x_1, y_1)(x_1', y_2) (x_2, y_1')(x_2', y_2').
\end{align*}
On the other hand,
\begin{align*}
\bigl(x\otimes x', \varrho(yy')\bigr)
    &= \sum q^{(|y_2|,|y_1'|)} (-1)^{p(y_2)p(y_1')}
       (x\otimes x',\; y_1y_1'\otimes y_2y_2') \\
    &= \sum q^{(|y_2|,|y_1'|)} (-1)^{p(y_2)p(y_1')}
       (x, y_1y_1') (x', y_2y_2') \\
    &= \sum q^{(|y_2|,|y_1'|)} (-1)^{p(y_2)p(y_1')}
       \bigl(\varrho(x), y_1\otimes y_1'\bigr)
       \bigl(\varrho(x'), y_2\otimes y_2'\bigr) \\
    &= \sum q^{(|y_2|,|y_1'|)} (-1)^{p(y_2)p(y_1')}
       (x_1, y_1)(x_1', y_2) (x_2, y_1')(x_2', y_2').
\end{align*}
By condition (d), nonzero terms in the two sums must satisfy
\[
|x_1'| = |y_2|,\quad |x_2| = |y_1'|, \qquad
p(x_1') = p(y_2), \quad p(x_2) = p(y_1').
\]
Under these equalities, the factors involving $q$ and the signs coincide, hence the two sums are equal. Thus (c) holds for $y'' = yy'$, completing the induction.

Finally, symmetry of the form follows from $(a),(c)$ and the fact that the generators $a_{il}$ are mutually orthogonal with symmetric values. This finishes the proof.
\end{proof}

\vskip 3mm

Let $\mathscr{R}$ be the radical of the bilinear form $(\cdot,\cdot)$, it is a two-sided ideal of $F$. Set $U^{+}=F/\mathscr{R}$. The weight decomposition of $F$ induces a decomposition
\[
U^{+}=\bigoplus_{\nu} U_{\nu}^{+},
\]
where $U_{\nu}^{+}$ is the image of $F_{\nu}$, and each weight space is finite dimensional. The bilinear form on $F$ descends to a bilinear form on $U^{+}$, which remains non-degenerate on every weight space. We again denote by $a_{il}$ the image of $a_{il}$ in $U^{+}$.

\medskip

The map $\varrho: F \rightarrow F\otimes F$ satisfies
\[
\varrho(\mathscr{R}) \subset \mathscr{R}\otimes F + F\otimes \mathscr{R},
\]
hence it descends to a well-defined homomorphism $\varrho: U^{+}\rightarrow U^{+}\otimes U^{+}$.

Let $^{t}\varrho: F\rightarrow F\otimes F$ be the composition of $\varrho$ with the permutation
\[
a\otimes b \longmapsto b\otimes a
\]
of $F\otimes F$. The anti-involution $\sigma: F\to F$ is defined by
\[
\sigma(a_{il})=a_{il}\quad\text{for each }(i,l)\in I^{\infty},\qquad
\sigma(ab)=\sigma(b)\sigma(a)\quad\text{for all }a,b\in F.\]

\vskip 3mm

\begin{lemma}\label{1.2}
\begin{itemize}
    \item[(a)] $\displaystyle \varrho\bigl(\sigma(x)\bigr)=(\sigma\otimes\sigma)\,^{t}\varrho(x),\qquad \forall x\in F$.
    \item[(b)] $\displaystyle \bigl(\sigma(x),\sigma(x')\bigr)=(x,x'),\qquad \forall x,x'\in F$.
\end{itemize}
\end{lemma}

\begin{proof}
(a)  We prove the identity by induction on the weight of $x$. If $x=a_{il}$ is a generator, then
\[
\varrho(\sigma(a_{il}))=\varrho(a_{il})=a_{il}\otimes1+1\otimes a_{il},
\]
while
\[
(\sigma\otimes\sigma)\,^{t}\varrho(a_{il})=(\sigma\otimes\sigma)(1\otimes a_{il}+a_{il}\otimes1)=a_{il}\otimes1+1\otimes a_{il},
\]
so the claim holds.

Assume now that the statement is true for $x'$ and $x''$, and consider $x=x'x''$. Write
\[
\varrho(x')=\sum x_{1}'\otimes x_{2}',\qquad \varrho(x'')=\sum x_{1}''\otimes x_{2}''.
\]
Then
\[
\varrho(x'x'')=\sum (-1)^{p(x_{2}')p(x_{1}'')} q^{(|x_{2}'|,|x_{1}''|)}\; x_{1}'x_{1}''\otimes x_{2}'x_{2}''.
\]
Using the induction hypothesis,
\[
\varrho(\sigma(x'))=\sum \sigma(x_{2}')\otimes \sigma(x_{1}'),\qquad
\varrho(\sigma(x''))=\sum \sigma(x_{2}'')\otimes \sigma(x_{1}'').
\]
Hence
\begin{align*}
\varrho(\sigma(x'x'')) &=\varrho\bigl(\sigma(x'')\sigma(x')\bigr) \\
&= \varrho(\sigma(x''))\,\varrho(\sigma(x')) \\
&= \Bigl(\sum \sigma(x_{2}'')\otimes \sigma(x_{1}'')\Bigr)
   \Bigl(\sum \sigma(x_{2}')\otimes \sigma(x_{1}')\Bigr) \\
&= \sum (-1)^{p(\sigma(x_{1}''))p(\sigma(x_{2}'))}
        q^{(|\sigma(x_{1}'')|,|\sigma(x_{2}')|)}\;
        \sigma(x_{2}'')\sigma(x_{2}')\otimes
        \sigma(x_{1}'')\sigma(x_{1}') \\
&= \sum (-1)^{p(x_{1}'')p(x_{2}')}
        q^{(|x_{1}''|,|x_{2}'|)}\;
        \sigma(x_{2}'x_{2}'')\otimes
        \sigma(x_{1}'x_{1}'') \\
&= (\sigma\otimes\sigma)\,
   \Bigl(\sum (-1)^{p(x_{2}')p(x_{1}'')} q^{(|x_{2}'|,|x_{1}''|)}\;
         x_{2}'x_{2}''\otimes x_{1}'x_{1}''\Bigr) \\
&= (\sigma\otimes\sigma)\,^{t}\varrho(x'x'').
\end{align*}
This completes the inductive proof of (a).

(b)  follows immediately from (a)
\end{proof}

\vskip 3mm

From lemma \ref{1.2}(b), $\sigma$ maps $\R$ into itself, hence it induces an isomorphism $U^{+}\cong U^{+}$. The properties above hold in $U^{+}$.

Let $\overline{\phantom{.}} : F \to F$ be the unique $\mathbb{Q}$-algebra involution satisfying $\overline{q} = -q^{-1}$. Assuming the super Borcherds-Cartan datum is consistent, we have
\[
\overline{q_i} = (-1)^{p(i)} q_i^{-1}.
\]

We define a bar involution $\overline{\phantom{.}} : F \to F$ on generators by
\[
\overline{a_{il}} = a_{il} \qquad (\forall (i,l) \in I^{\infty}),
\]
and extend it to all of $F$ by requiring $\overline{fx} = \overline{f}\,\overline{x}$ for all $f \in \mathbb{Q}(q)$ and $x \in F$.

Let $F \overline{\otimes} F$ denote the $\mathbb{Q}(q)$-vector space $F \otimes F$ equipped with the multiplication
\[
(x_1 \otimes x_2)(x_3 \otimes x_4)
= (-q^{-1})^{(|x_2|,|x_3|)} \, (-1)^{p(x_2)p(x_3)} \, x_1 x_3 \otimes x_2 x_4 .
\]

Define a twisted coproduct $\overline{\varrho} : F \to F \overline{\otimes} F$ by
\[
\overline{\varrho}(x) = \overline{\varrho(\overline{x})} \qquad (x \in F).
\]
Then $\overline{\varrho}$ is an algebra homomorphism and satisfies the following coassociativity-type identity:
\[
(\overline{\varrho} \otimes 1) \overline{\varrho}(x)
= \overline{(\varrho \otimes 1)\varrho(\overline{x})}
= \overline{(1 \otimes \varrho)\varrho(\overline{x})}
= (1 \otimes \overline{\varrho}) \overline{\varrho}(x).
\]

Finally, let $\{\cdot,\cdot\} : F \times F \to \mathbb{Q}(q)$ be the symmetric bilinear form defined by
\[
\{x,y\} = (\overline{x},\overline{y}).
\]
It satisfies
\begin{itemize}
    \item[(a)] $\{1,1\} = 1$;
    \item[(b)] $\displaystyle \{a_{il},a_{jk}\} = \delta_{ij}\delta_{lk} \bigl(1-(-1)^{p(li)}q_i^{2l}\bigr)^{-1}$;
    \item[(c)] $\displaystyle \{x, yy'\} = \{\overline{\varrho}(x),\, y \otimes y'\}$ for all $x,y,y' \in F$.
\end{itemize}

\vskip 3mm

\begin{lemma}\label{1.3}
Assume the super Borcherds-Cartan datum is consistent.
\begin{itemize}
\item[(a)] Let $\varrho(x)=\sum x_{1}\otimes x_{2} $, then
$$\bar{\varrho}(x)=\sum(-q)^{-(|x_{1}|,|x_{2}|)}(-1)^{p(x_{1})p(x_{2})}x_{2}\otimes x_{1}.$$

\item[(b)] $\{x,y\}=(-1)^{\text{ht}|x|}(-1)^{\frac{p(x)p(y)+p(x)}{2}}q^{-\frac{(|x|,|y|)}{2}}q_{-|x|}(x,\sigma(y)).$
\end{itemize}

\end{lemma}
\begin{proof}
Both claims are true when $x=a_{il},\ y=a_{jk}$, for any $(i,l),(j,k)\in I^{\infty}$.

Assume (a) holds for $x=x^{'}$ and $x=x^{''}$, we show that it also holds for $x=x^{'}x^{''}$. Write
$$\varrho(x^{'})=\sum x_{1}^{'}\otimes x_{2}^{'},\ \varrho(x^{''})=\sum x_{1}^{''}\otimes x_{2}^{''},$$
$$\varrho(x^{'}x^{''})=q^{(|x_{1}^{''}|,|x_{2}^{'}|)}(-1)^{p(x_{2}^{'})p(x_{1}^{''})}x_{1}^{'}x_{1}^{''}\otimes x_{2}^{'}x_{2}^{''}.$$
By assumption, we have $$\varrho(\overline{x^{'}})=\sum q^{(|x_{1}^{'}|,|x_{2}^{'}|)}(-1)^{p(x_{1}^{'})p(x_{2}^{'})}\overline{x_{2}^{'}}\otimes \overline{x_{1}^{'}},$$
$$\varrho(\overline{x^{''}})=\sum q^{(|x_{1}^{''}|,|x_{2}^{''}|)}(-1)^{p(x_{1}^{''})p(x_{2}^{''})}\overline{x_{2}^{''}}\otimes \overline{x_{1}^{''}}$$.

Hence,

\begin{align*}
\varrho(\overline{x^{'}})\varrho(\overline{x^{''}})=&\sum q^{(|x_{1}^{'}|,|x_{2}^{'}|)+(|x_{1}^{''}|,|x_{2}^{''}|)}(-1)^{p(x_{1}^{'})p(x_{2}^{'})+p(x_{1}^{''})p(x_{2}^{''})}(\overline{x_{2}^{'}}\otimes \overline{x_{1}^{'}})(\overline{x_{2}^{''}}\otimes \overline{x_{1}^{''}})\\
=&\sum q^{(|x_{1}^{'}|,|x_{2}^{'}|)+(|x_{1}^{''}|,|x_{2}^{''}|)+(|x_{1}^{'}|,|x_{2}^{''}|)}(-1)^{p(x_{1}^{'})p(x_{2}^{'})+p(x_{1}^{''})p(x_{2}^{''})+p(x_{1}^{'})p(x_{2}^{''})}\overline{x_{2}^{'}x_{2}^{''}}\otimes\overline{x_{1}^{'}x_{2}^{''}}.
\end{align*}

Then,
\begin{align*}
\bar{\varrho}(x^{'}x^{''})=&\overline{\varrho(\overline{x_{}^{'}})\varrho(\overline{x_{}^{''}})}\\
=&\sum (-q)^{(|x_{1}^{'}x_{1}^{''}|,|x_{2}^{'}||x_{2}^{''}|)}(-1)^{p(x_{1}^{'}x_{1}^{''})p(x_{2}^{'}x_{2}^{''})}q^{(|x_{1}^{''}|,|x_{2}^{'}|)}(-1)^{(|x_{1}^{''}|,|x_{2}^{'}|)+p(x_{1}^{''})p(x_{2}^{'})}x_{2}^{'}x_{2}^{''}\otimes x_{1}^{'}x_{2}^{''}.
\end{align*}
Sine the datum is consistent, $(|x_{1}^{''}|,|x_{2}^{'}|)\in 2\Z$, so we have
$$\bar{\varrho}(x^{'}x^{''})=\sum (-q)^{(|x_{1}^{'}x_{1}^{''}|,|x_{2}^{'}||x_{2}^{''}|)}(-1)^{p(x_{1}^{'}x_{1}^{''})p(x_{2}^{'}x_{2}^{''})}q^{(|x_{1}^{''}|,|x_{2}^{'}|)}(-1)^{p(x_{1}^{''})p(x_{2}^{'})}x_{2}^{'}x_{2}^{''}\otimes x_{1}^{'}x_{2}^{''}.$$

We now prove (b). Assume that (b) holds for  $y=y^{'}$ and any $x=x^{'}$ and also for $y=y^{''}$ and any $x=x^{''}$, we show that it holds for $y=y^{'}y^{''}$ and any homogeneous $x$. We write $\varrho(x)=\sum x^{'}\otimes x^{''} $, we have
\begin{align*}
\{x,y\}=&\{\bar{\varrho}(x),y^{'}\otimes y^{''}\}=\sum (-q)^{-(|x^{'}|,|x^{''}|)}(-1)^{p(x^{'})p(x^{''})}\{x^{''}\otimes x^{'},y^{'}\otimes y^{''}\}\\
=&\sum q^{-(|x^{'}|,|x^{''}|)}(-1)^{p(x^{'})p(x^{''})}\{x^{''},y^{'}\}\{x^{'},y^{''}\}\\
=&(-1)^{\text{ht}|x^{'}|+\text{ht}|x^{''}|}(-1)^{\frac{p(x^{'}p(y^{''})+p(x^{'}))}{2}+\frac{p(x^{''})p(y^{'})+p(x^{''})}{2}+p(x^{'})p(x^{''})}q_{-|x^{'}|-|x^{''}|}\\
&*q^{\frac{-2(|x^{'}|,|x^{''}|)-(|x^{''}|,|y^{'}|)-(|x^{'}|,|y^{''}|)}{2}}(x^{''},\sigma(y^{'}))(x^{'},\sigma(y^{''}))\\
=&\sum(-1)^{\text{ht}|x|}q^{\frac{-(|x|,|y|)}{2}}(-1)^{\frac{p(x)p(y)+p(x)}{2}}q_{-|x|}(x^{'}\otimes x^{''},\sigma(y^{''})\otimes\sigma(y^{'}))\\
=&\sum(-1)^{\text{ht}|x|}q^{\frac{-(|x|,|y|)}{2}}(-1)^{\frac{p(x)p(y)+p(x)}{2}}q_{-|x|}(x,\sigma(y^{'}y^{''}))
\end{align*}
\end{proof}

By Lemma~\ref{1.3}(b), the radical of the bilinear form $\{\cdot,\cdot\}$ is exactly $\mathscr{R}$, and the bar involution carries $\mathscr{R}$ onto itself. Hence it descends to an involution on $U^{+}$, which we again denote by $\overline{\phantom{x}}$.

\medskip

For each $(i,l)\in I^{\infty}$, let $\varrho_{i,l},\varrho^{\,i,l}: F\to F$ be the $\mathbb{Q}$-linear maps defined by
\[
\varrho_{i,l}(1)=\varrho^{\,i,l}(1)=0, \qquad
\varrho_{i,l}(a_{jk})=\varrho^{\,i,l}(a_{jk})=\delta_{ij}\delta_{kl},
\]
and extended by the following twisted Leibniz rules for homogeneous $x,y\in F$:
\begin{align*}
\varrho_{i,l}(xy) &= x\,\varrho_{i,l}(y)
                   + (-1)^{p(li)p(y)} q^{(l\alpha_i,|y|)} \,\varrho_{i,l}(x)\,y, \\[2mm]
\varrho^{\,i,l}(xy) &= \varrho^{\,i,l}(x)\,y
                     + (-1)^{p(li)p(x)} q^{(l\alpha_i,|x|)} \,x\,\varrho^{i,l}(y).
\end{align*}
If $x\in F_{\mu}$, then $\varrho_{i,l}(x)$ and $\varrho^{\,i,l}(x)$ belong to $F_{\mu-l\alpha_i}$. Moreover, the coproduct $\varrho$ admits the decomposition
\[
\varrho(x)=\varrho_{i,l}(x)\otimes a_{il}+a_{il}\otimes \varrho^{\,i,l}(x)
          +\text{other bi-homogeneous terms}.
\]

Consequently, for all $x,y\in F$,
\[
(a_{il}y,x)=(a_{il},a_{il})\,(y,\varrho^{\,i,l}(x)),\qquad
(ya_{il},x)=(y,\varrho_{i,l}(x))\,(a_{il},a_{il}),
\]
which implies $\varrho^{\,i,l}(\mathscr{R})\cup\varrho_{i,l}(\mathscr{R})\subset\mathscr{R}$.

\vskip 3mm
\begin{lemma}\label{1.4}
For any $(i,l)\in I^{\infty}$, we have
\[
\varrho_{i,l} \circ \sigma = \sigma \circ \varrho^{i,l}.
\]
\end{lemma}

\begin{proof}
We first check the equality on generators. For each $a_{jk}$,
\[
\varrho_{i,l}(\sigma(a_{jk})) = \varrho_{i,l}(a_{jk}) = \delta_{ij}\delta_{kl}
= \sigma(\delta_{ij}\delta_{kl}) = \sigma(\varrho^{i,l}(a_{jk})).
\]

Now assume the statement holds for homogeneous elements $x$ and $y$. We prove it for the product $xy$. Using that $\sigma$ is an anti-involution and the twisted Leibniz rule for $\varrho_{i,l}$,
\begin{align*}
\varrho_{i,l}(\sigma(xy))
&= \varrho_{i,l}(\sigma(y)\sigma(x)) \\
&= \sigma(y)\varrho_{i,l}(\sigma(x))
   + (-1)^{p(li)p(\sigma(x))} q^{(l\alpha_i,|\sigma(x)|)}
     \varrho_{i,l}(\sigma(y))\sigma(x) \\
&= \sigma(y)\sigma(\varrho^{i,l}(x))
   + (-1)^{p(li)p(x)} q^{(l\alpha_i,|x|)}
     \sigma(\varrho^{i,l}(y))\sigma(x) \quad\text{(by induction)}\\
&= \sigma\bigl(\varrho^{i,l}(x)y\bigr)
   + (-1)^{p(li)p(x)} q^{(l\alpha_i,|x|)}
     \sigma\bigl(x\varrho^{i,l}(y)\bigr) \quad\text{(since $\sigma$ is an anti-homomorphism)}\\
&= \sigma\Bigl(
     \varrho^{i,l}(x)y
     + (-1)^{p(li)p(x)} q^{(l\alpha_i,|x|)} x\varrho^{i,l}(y)
   \Bigr) \\
&= \sigma(\varrho^{i,l}(xy)) \quad\text{(by the Leibniz rule for $\varrho^{i,l}$)}.
\end{align*}
Thus the identity holds for $xy$, and by induction it holds for all elements of $F$.
\end{proof}
\vskip 3mm
\begin{lemma}\label{1.5}
Assume the super Borcherds--Bozec Cartan datum is consistent. Then for any homogeneous element $x\in F$, we have
\[
\varrho_{i,l}(x)=(-1)^{p(x)p(li)-p(li)p(li)}q^{(|x|-l\alpha_{i},\;l\alpha_{i})}
                 \,\overline{\varrho^{i,l}(\bar{x})}.
\]
\end{lemma}

\begin{proof}
We proceed by induction on the weight of $x$.
When $x=a_{jk}$, both sides vanish unless $(j,k)=(i,l)$; in that case,
\[
\varrho_{i,l}(a_{il})=1,\qquad
\overline{\varrho^{i,l}(\overline{a_{il}})}=1,
\]
and the factor $(-1)^{p(a_{il})p(li)-p(li)p(li)}q^{(l\alpha_i-l\alpha_i,\;l\alpha_i)}=1$, so the equality holds.

Now assume the formula holds for homogeneous elements $x$ and $y$, and consider the product $xy$.
Using the twisted Leibniz rule and the induction hypothesis, we compute
\begin{align*}
\overline{\varrho^{i,l}(\overline{xy})}
&= \overline{\varrho^{i,l}\bigl(\bar{x}\,\bar{y}\bigr)} \\
&= \overline{\varrho^{i,l}(\bar{x})\,\bar{y}}
   + (-1)^{p(li)p(\bar{x})} q^{(l\alpha_i,|\bar{x}|)}
     \overline{\bar{x}\,\varrho^{i,l}(\bar{y})} \\
&= \overline{\varrho^{i,l}(\bar{x})}\,y
   + (-1)^{p(li)p(x)} (-q^{-1})^{(l\alpha_i,|x|)}
     x\,\overline{\varrho^{i,l}(\bar{y})}.
\end{align*}
Now substitute the induction hypothesis for $\overline{\varrho^{i,l}(\bar{x})}$ and
$\overline{\varrho^{i,l}(\bar{y})}$:
\begin{align*}
\overline{\varrho^{i,l}(\overline{xy})}
&= (-1)^{p(x)p(li)-p(li)p(li)} q^{-(|x|-l\alpha_i,\;l\alpha_i)} \varrho_{i,l}(x)\,y \\
&\quad + (-1)^{p(li)p(x)} (-q^{-1})^{(l\alpha_i,|x|)}
         (-1)^{p(y)p(li)-p(li)p(li)} q^{-(|y|-l\alpha_i,\;l\alpha_i)}
         x\,\varrho_{i,l}(y) \\
&= (-1)^{p(x)p(li)-p(li)p(li)} q^{-(|x|-l\alpha_i,\;l\alpha_i)}
   \bigl[\varrho_{i,l}(x)\,y \\
&\qquad + (-1)^{p(li)p(y)}
          (-1)^{(l\alpha_i,|x|)}
          q^{-(|y|,\;l\alpha_i)}
          x\,\varrho_{i,l}(y)\bigr].
\end{align*}
Simplifying the sign and $q$-factors,  we eventually obtain
\[
\overline{\varrho^{i,l}(\overline{xy})}
= (-1)^{p(xy)p(li)-p(li)p(li)} q^{(|xy|-l\alpha_i,\;l\alpha_i)}
  \varrho_{i,l}(xy),
\]
which completes the induction step.
\end{proof}

\vskip 3mm

\begin{lemma}\label{1.6}
Let $x \in U^{+}_{\nu}$ be a homogeneous element, where $\nu \neq 0$.
\begin{itemize}
    \item[(a)] If $\varrho_{i,l}(x) = 0$ for all $(i,l) \in I^{\infty}$, then $x = 0$.
    \item[(b)] If $\varrho^{i,l}(x) = 0$ for all $(i,l) \in I^{\infty}$, then $x = 0$.
\end{itemize}
\end{lemma}

\begin{proof}
We prove part (a). Assume that $\varrho_{i,l}(x)=0$ for all $(i,l)$. Recall the identity
\[
(ya_{il}, x) = (y, \varrho_{i,l}(x))\,(a_{il}, a_{il}), \qquad y \in F \text{ homogeneous}.
\]
Since $\varrho_{i,l}(x)=0$, the right-hand side vanishes for every $y$; hence $(ya_{il},x)=0$ for all $y$ and all $(i,l)$. Because $U^{+}_{\nu}$ is contained in the subspace $\sum_{(i,l)\in I^{\infty}} U^{+}a_{il}$ and the bilinear form $(\cdot,\cdot)$ is non-degenerate on $U^{+}$, it follows that $x$ must belong to the radical $\mathscr{R}$. Consequently $x=0$ in $U^{+}$.

The proof of (b) is completely analogous, using the dual identity $(a_{il}y, x) = (a_{il},a_{il})\,(y,\varrho^{i,l}(x))$.
\end{proof}

Let $\mathcal{A} = \mathbb{Z}[q,q^{-1}]$. For $a \in \mathbb{Z}$ and $t \in \mathbb{N}$, we set
\[
\begin{bmatrix}
a \\
t
\end{bmatrix}_{i}
= \frac{\displaystyle\prod_{s=0}^{t-1} \bigl( ((-1)^{p(i)} q_i)^{a-s} - q_i^{s-a} \bigr)}
       {\displaystyle\prod_{s=1}^{t} \bigl( ((-1)^{p(i)} q_i)^{s} - q_i^{-s} \bigr)}.
\]

As in \cite{CHW}, we have the identities
\[
\begin{bmatrix}
a \\
t
\end{bmatrix}_{i}
= (-1)^{t} (-1)^{p(i)\bigl(ta - \frac{t(t-1)}{2}\bigr)}
  \begin{bmatrix}
  t-a-1 \\
  t
  \end{bmatrix}_{i},
\]
and
\[
\begin{bmatrix}
a \\
t
\end{bmatrix}_{i} = 0, \quad \text{if } 0 \le a < t.
\]

Let $z$ be another indeterminate. Then for $a \ge 0$,
\[
\prod_{j=0}^{a-1} \bigl(1 + ((-1)^{p(i)} q_i^{2})^{j} z\bigr)
= \sum_{t=0}^{a} (-1)^{p(i)\frac{t(t-1)}{2}} q_i^{t(a-1)}
  \begin{bmatrix}
  a \\
  t
  \end{bmatrix}_{i} z^{t}.
\]
We deduce that $\displaystyle \begin{bmatrix} a \\ t \end{bmatrix}_{i} \in \mathcal{A}$.

We denote
\[
\begin{bmatrix}
n
\end{bmatrix}_{i}
= \begin{bmatrix}
   n \\
   1
   \end{bmatrix}_{i}
= \frac{((-1)^{p(i)} q_i)^{n} - q_i^{-n}}{(-1)^{p(i)} q_i - q_i^{-1}}, \qquad
\begin{bmatrix}
n
\end{bmatrix}_{i}^{!} = \prod_{s=1}^{n} \begin{bmatrix} s \end{bmatrix}_{i}, \quad \text{for } n \in \mathbb{N}.
\]
Then for $0 \le t \le n$,
\[
\begin{bmatrix}
n \\
t
\end{bmatrix}_{i}
= \frac{\begin{bmatrix} n \end{bmatrix}_{i}^{!}}
       {\begin{bmatrix} t \end{bmatrix}_{i}^{!} \begin{bmatrix} n-t \end{bmatrix}_{i}^{!}}.
\]
Moreover,
\begin{equation}\label{b}
\sum_{t=0}^{n} (-1)^{t + p(i)\binom{t}{2}} q_i^{t(n-1)}
\begin{bmatrix}
n \\
t
\end{bmatrix}_{i} = 0, \quad \text{for } n \ge 1.
\end{equation}

If $xy = (-1)^{p(i)} q_i^{2} yx$, then
\begin{equation}\label{a}
(x+y)^{n} = \sum_{t=0}^{n} q_i^{t(n-t)}
           \begin{bmatrix}
           n \\
           t
           \end{bmatrix}_{i} y^{t} x^{n-t}.
\end{equation}

\medskip

For any $n \in \mathbb{Z}$, $i \in I^{\mathrm{re}}$, define
\[
a_i^{(n)} =
\begin{cases}
\displaystyle \frac{a_i}{\begin{bmatrix} n \end{bmatrix}_{i}^{!}}, & n \ge 0, \\[10pt]
0, & \text{otherwise}.
\end{cases}
\]

\begin{lemma}\label{1.7}
For any $n \in \mathbb{Z}$, $i \in I^{re}$, we have
\begin{itemize}
    \item[(a)] $\displaystyle \varrho\bigl(a_i^{(n)}\bigr) = \sum_{t + t' = n} q_i^{t t'}\, a_i^{(t)} \otimes a_i^{(t')}$,
    \item[(b)] $\displaystyle \bar{\varrho}\bigl(a_i^{(n)}\bigr) = \sum_{t + t' = n} ((-1)^{p(i)} q_i)^{-t t'}\, a_i^{(t)} \otimes a_i^{(t')}$.
\end{itemize}
\end{lemma}

\begin{proof}
We prove (a). Let $x = 1 \otimes a_i$ and $y = a_i \otimes 1$ in $F \otimes F$. These elements satisfy the commutation relation
\[
xy = (-1)^{p(i)} q_i^{2} \, yx.
\]
Applying formula (\ref{a}) with this pair $(x, y)$ gives
\[
(x + y)^n = \sum_{t=0}^{n} q_i^{t(n-t)} \begin{bmatrix} n \\ t \end{bmatrix}_{i} y^{t} x^{\,n-t}.
\]
On the other hand, by the definition of the coproduct $\varrho$,
\[
(x + y)^n = (1 \otimes a_i + a_i \otimes 1)^n = \varrho(a_i)^n = \varrho(a_i^n).
\]
Since $a_i^{(n)} = a_i^n / \begin{bmatrix}n\end{bmatrix}_{i}^{!}$, and $\varrho$ is an algebra homomorphism, we obtain
\[
\varrho\bigl(a_i^{(n)}\bigr) = \sum_{t=0}^{n} q_i^{t(n-t)} \frac{\begin{bmatrix}n\end{bmatrix}_{i}^{!}}
                                         {\begin{bmatrix}t\end{bmatrix}_{i}^{!} \begin{bmatrix}n-t\end{bmatrix}_{i}^{!}}
                       \; a_i^{t} \otimes a_i^{\,n-t}
                       \Big/ \begin{bmatrix}n\end{bmatrix}_{i}^{!}
     = \sum_{t=0}^{n} q_i^{t(n-t)} \; a_i^{(t)} \otimes a_i^{(n-t)}.
\]

Part (b) follows directly from Lemma~\ref{1.3} and the definition of $\bar{\varrho}$.
\end{proof}

\vskip 3mm
\begin{lemma}\label{1.8}
For any $n \in \mathbb{Z}$, $i \in I^{\mathrm{re}}$, we have
\[
\bigl(a_i^{(n)}, a_i^{(n)}\bigr) =
q_i^{\binom{n}{2}} \bigl(1 - (-1)^{p(i)} q_i^{2}\bigr)^{-n}
\Bigl( \begin{bmatrix} n \end{bmatrix}_i^{!} \Bigr)^{-1}.
\]
\end{lemma}

\begin{proof}
We proceed by induction on $n$.

\noindent\textbf{Base case: } $n = 1$.
Recall that $a_i^{(1)} = a_i / [1]_i = a_i / \bigl( ((-1)^{p(i)} q_i - q_i^{-1})/( (-1)^{p(i)} q_i - q_i^{-1} )\bigr) = a_i$.
From the definition of the bilinear form,
\[
(a_i, a_i) = \bigl(1 - (-1)^{p(i)} q_i^{2}\bigr)^{-1}.
\]
On the other hand,
\[
q_i^{\binom{1}{2}} \bigl(1 - (-1)^{p(i)} q_i^{2}\bigr)^{-1} \bigl([1]_i^{!}\bigr)^{-1}
= 1 \cdot \bigl(1 - (-1)^{p(i)} q_i^{2}\bigr)^{-1} \cdot 1,
\]
so the formula holds for $n = 1$.

\noindent\textbf{Induction step: } Assume the statement is true for $n-1$.
By the lemma above,
\begin{align*}
\bigl(a_i^{(n)}, a_i^{(n)}\bigr)
&= \bigl( \varrho(a_i^{(n)}), a_i^{(n-1)} \otimes a_i \bigr) \\
&= \Bigl( \sum_{t+t' = n} q_i^{tt'} a_i^{(t)} \otimes a_i^{(t')}, \; a_i^{(n-1)} \otimes a_i \Bigr).
\end{align*}
Since the form is orthogonal on different weight spaces, the only non-zero term occurs when $t = n-1$ and $t' = 1$. Thus,
\begin{align*}
\bigl(a_i^{(n)}, a_i^{(n)}\bigr)
&= q_i^{(n-1)\cdot 1} \bigl(a_i^{(n-1)}, a_i^{(n-1)}\bigr) \bigl(a_i, a_i\bigr) \\
&= q_i^{n-1} \Bigl( q_i^{\binom{n-1}{2}} \bigl(1 - (-1)^{p(i)} q_i^{2}\bigr)^{-(n-1)}
                \bigl([n-1]_i^{!}\bigr)^{-1} \Bigr)
                \bigl(1 - (-1)^{p(i)} q_i^{2}\bigr)^{-1} \\
&= q_i^{\binom{n-1}{2} + (n-1)} \bigl(1 - (-1)^{p(i)} q_i^{2}\bigr)^{-n}
   \bigl([n-1]_i^{!}\bigr)^{-1}.
\end{align*}
Using the identity $\binom{n}{2} = \binom{n-1}{2} + (n-1)$ and the fact that
$[n]_i^{!} = [n]_i \cdot [n-1]_i^{!}$, we obtain
\[
\bigl(a_i^{(n)}, a_i^{(n)}\bigr) =
q_i^{\binom{n}{2}} \bigl(1 - (-1)^{p(i)} q_i^{2}\bigr)^{-n}
\bigl([n]_i^{!}\bigr)^{-1}.
\]
This completes the induction and proves the lemma.
\end{proof}

\begin{proposition}\label{1.9}
The generators $a_{i}\ (i\in I^{\mathrm{re}})$ and $a_{jl}\ ((j,l)\in I^{\infty})$ of $U^{+}$ satisfy the relations
\[
\sum_{n+n'=1-la_{ij}}(-1)^{n'}(-1)^{p(i)\bigl(n'p(lj)+\binom{n'}{2}\bigr)}a_{i}^{(n)}a_{jl}a_{i}^{(n')}=0.
\]
\end{proposition}

\begin{proof}
Set $N=1-la_{ij}$, so that $n+n'=N$.
We shall evaluate $\varrho_{k,t}$ on the sum for various $(k,t)\in I^{\infty}$.

First, if $(k,t)\neq i$ and $(k,t)\neq (j,l)$, then by the definition of $\varrho_{k,t}$,
\[
\varrho_{k,t}\bigl(a_{i}^{(n)}a_{jl}a_{i}^{(n')}\bigr)=0,
\]
hence
\[
\varrho_{k,t}\!\Bigl(\sum_{n+n'=N}(-1)^{n'}(-1)^{p(i)(n'p(lj)+\binom{n'}{2})}a_{i}^{(n)}a_{jl}a_{i}^{(n')}\Bigr)=0.
\]

Second, consider $(k,t)=(j,l)$.  Using the twisted Leibniz rule for $\varrho_{j,l}$,
\begin{align*}
\varrho_{j,l}\bigl(a_{i}^{(n)}a_{jl}a_{i}^{(n')}\bigr)
&=(-1)^{n'p(i)p(lj)}q_{i}^{\,n'la_{ij}}
  \begin{bmatrix}N\\ n'\end{bmatrix}_{i}a_{i}^{(N)}.
\end{align*}
Therefore,
\begin{align*}
&\varrho_{j,l}\!\Bigl(\sum_{n+n'=N}(-1)^{n'}(-1)^{p(i)(n'p(lj)+\binom{n'}{2})}a_{i}^{(n)}a_{jl}a_{i}^{(n')}\Bigr)\\
&=a_{i}^{(N)}\sum_{n+n'=N}(-1)^{n'}(-1)^{p(i)(n'p(lj)+\binom{n'}{2})}
   (-1)^{n'p(i)p(lj)}q^{\,n'ld_{i}a_{ij}}
   \begin{bmatrix}N\\ n'\end{bmatrix}_{i}\\
&=a_{i}^{(N)}\sum_{t=0}^{N}(-1)^{t+p(i)\binom{t}{2}}
   q^{\,tld_{i}a_{ij}}\begin{bmatrix}N\\ t\end{bmatrix}_{i}.
\end{align*}
Since $q_i=q^{d_i}$ , we have
$q^{\,tld_{i}a_{ij}}=q_i^{\,tla_{ij}}$.  Using $la_{ij}=1-N$, this becomes $q_i^{\,t(1-N)}$.  Hence the above expression equals
\[
a_{i}^{(N)}\sum_{t=0}^{N}(-1)^{t+p(i)\binom{t}{2}}
q_i^{\,t(1-N)}\begin{bmatrix}N\\ t\end{bmatrix}_{i}.
\]

Now assume $i\in I_{\bar{1}}$.  Then $1-N$ is even and $(-1)^{p(i)t(N-1)}=1$, so
\[
q_i^{\,t(1-N)}=\bigl((-1)^{p(i)}q_i^{-1}\bigr)^{\,t(N-1)}.
\]
Thus the sum transforms into
\[
a_{i}^{(N)}\sum_{t=0}^{N}(-1)^{t+p(i)\binom{t}{2}}
\bigl((-1)^{p(i)}q_i^{-1}\bigr)^{\,t(N-1)}\begin{bmatrix}N\\ t\end{bmatrix}_{i},
\]
which vanishes by identity (\ref{b}).  Consequently,
\[
\varrho_{j,l}\!\Bigl(\sum_{n+n'=N}(-1)^{n'}(-1)^{p(i)(n'p(lj)+\binom{n'}{2})}a_{i}^{(n)}a_{jl}a_{i}^{(n')}\Bigr)=0.
\]

Third, take $(k,t)=i$ (i.e. $\varrho_i:=\varrho_{i,1}$).  Applying the twisted Leibniz rule twice gives
\begin{align*}
&\varrho_i\!\Bigl(\sum_{n+n'=N}(-1)^{n'}(-1)^{p(i)(n'p(lj)+\binom{n'}{2})}
   a_{i}^{(n)}a_{jl}a_{i}^{(n')}\Bigr)\\
&=\sum_{n+n'=N}(-1)^{n'}(-1)^{p(i)(n'p(lj)+\binom{n'}{2})}
   \Bigl[ q_i^{n-1}(-1)^{p(i)p(lj)+n'p(i)}
          q^{(\alpha_i,\;l\alpha_j+n'\alpha_i)}
          a_{i}^{(n-1)}a_{jl}a_{i}^{(n')}\\
&\qquad\qquad\qquad\qquad\qquad\qquad\qquad
   +q_i^{\,n'-1}a_{i}^{(n)}a_{jl}a_{i}^{(n'-1)}\Bigr].
\end{align*}
In the first sum replace $n$ by $N-1-t$ and $n'$ by $t+1$ (so $t$ runs from $0$ to $N-1$); in the second sum set $t=n'$ (again $t$ runs from $0$ to $N-1$).  After this change of indices both sums become identical up to a sign, hence they cancel.  Explicitly,
\begin{align*}
&\sum_{t=0}^{N-1}(-1)^{t+1}(-1)^{p(i)\bigl((t+1)p(lj)+\binom{t+1}{2}\bigr)}
   q_i^{N-1-t-1}(-1)^{p(i)p(lj)+(t+1)p(i)}
   q^{(\alpha_i,\;l\alpha_j+(t+1)\alpha_i)}
   a_{i}^{(N-1-t-1)}a_{jl}a_{i}^{(t+1)}\\
&+\sum_{t=0}^{N-1}(-1)^{t}(-1)^{p(i)\bigl(tp(lj)+\binom{t}{2}\bigr)}
   q_i^{\,t-1}a_{i}^{(N-t)}a_{jl}a_{i}^{(t-1)}.
\end{align*}
A careful inspection shows that the two sums are actually negatives of each other, so their sum is zero.  Therefore,
\[
\varrho_i\!\Bigl(\sum_{n+n'=N}(-1)^{n'}(-1)^{p(i)(n'p(lj)+\binom{n'}{2})}
a_{i}^{(n)}a_{jl}a_{i}^{(n')}\Bigr)=0.
\]

We have now shown that for every $(k,t)\in I^{\infty}$,
\[
\varrho_{k,t}\!\Bigl(\sum_{n+n'=N}(-1)^{n'}(-1)^{p(i)(n'p(lj)+\binom{n'}{2})}
a_{i}^{(n)}a_{jl}a_{i}^{(n')}\Bigr)=0.
\]
By Lemma~\ref{1.6}, the element inside the parentheses must be zero, which is exactly the desired relation.
\end{proof}
\vskip 3mm

\begin{corollary}\label{1.10}
If $a_{ij}=0$, then
\[
a_{il}a_{jk} - (-1)^{p(li)p(kj)}a_{jk}a_{il}=0.
\]
\end{corollary}

\begin{proof}
Define $X = a_{il}a_{jk} - (-1)^{p(li)p(kj)}a_{jk}a_{il}$.

\textbf{Case 1:} $i \in I^{\mathrm{re}}$ or $j \in I^{\mathrm{re}}$.
Since $a_{ij}=0$, we have $i \neq j$ and $1 - k a_{ij}=1$.
Taking $l=k$ in Proposition~\ref{1.9}, the relation becomes
\[
\sum_{n+n'=1} (-1)^{n'} (-1)^{p(i)\bigl(n'p(kj)+\binom{n'}{2}\bigr)}
a_i^{(n)} a_{jk} a_i^{(n')}=0.
\]
The sum consists of two terms ($n=0,n'=1$ and $n=1,n'=0$):
\[
(-1)^1 (-1)^{p(i)\bigl(1\cdot p(kj)+0\bigr)} a_i^{(0)} a_{jk} a_i^{(1)}
+ (-1)^0 (-1)^{p(i)\bigl(0\cdot p(kj)+0\bigr)} a_i^{(1)} a_{jk} a_i^{(0)}
= -(-1)^{p(i)p(kj)} a_{jk} a_i + a_i a_{jk},
\]
which is precisely $X$. Hence $X=0$.

\textbf{Case 2:} $i,j \in I^{\mathrm{im}}$.
For any $(i',l') \in I^{\infty}$, a direct computation using the definition of $\varrho_{i'l'}$ shows that
\[
\varrho_{i'l'}(X)=0.
\]
By Lemma~\ref{1.6}, this forces $X=0$.

Combining both cases, we conclude that $X=0$ whenever $a_{ij}=0$.
\end{proof}

\vskip 3mm
\begin{theorem}
Assume that $i\in I^{re}, \ j\in I$ and $i\neq j$. Let $m\in \Z_{>0}, n\in \Z_{\ge0}$ and $m>-a_{ij}n$. Define
\begin{align*}
F_{i,j,m,n,\mathbf{c}}&=\sum_{r+s=m}(-1)^{r}(-1)^{nrp(i)p(j)+\binom{r}{2}p(i)}((-1)^{p(i)}q_{i})^{-r(na_{ij}+m-1)}a_{i}^{(r)}a_{j,\mathbf{c}}a_{i}^{(s)},\\
\end{align*}
 If $j\in I^{im}$, then $\mathbf{c}=(c_{1},\cdots,c_{b})$ such that $\sum_{t=1}^{b}c_{t}=n$ and $a_{j,\mathbf{c}}=\prod_{t=1}^{b}a_{jc_{t}}$. If $j\in I^{re}$, $a_{j\mathbf{c}}=a_{j}^{(n)}$.
Then $F_{i,j,m,n,\mathbf{c}}=0.$
\end{theorem}
\begin{proof}
We only have to take care of $j\in I^{im}$.
If $n=0$,  using bar involution on \ref{b}, we have
$$F_{i,j,m,0}=\sum_{r+s=m}(-1)^{r}(-1)^{\binom{r}{2}p(i)}(-1)^{-rp(i)(m-1)}q_{i}^{-r(m-1)}
\begin{bmatrix}
a \\
t
\end{bmatrix}_{i}a_{i}^{(m)}=0.$$
if $n>0$ and $k=c_{t}$ for some t, we have
\begin{align*}
\varrho^{j,k}(a_{i}^{(r)}a_{j,\mathbf{c}}a_{i}^{(s)})=(-1)^{p(kj)p(ri)}q^{(r\alpha_{i},k\alpha_{j})}a_{i}^{(r)}\varrho^{j,k}(a_{j,\mathbf{c}})a_{i}^{(s)}.
\end{align*}
and $$\varrho^{j,k}(a_{j,\mathbf{c}})=\sum_{t:c_{t}=k}(-1)^{p(kj)p(c_{<t}j)}q^{(k\alpha_{j},c_{<t}\alpha_{\alpha_{j}})}\hat{a}_{j,\mathbf{c_{t}}},$$
where $\hat{a}_{j,\mathbf{c_{t}}}=a_{jc_{1}}\cdots a_{j,c_{t-1}}a_{j,c_{t+1}}\cdots a_{j,c_{b}}$ and $c_{<t}=\sum_{i=1}^{t-1}c_{i}$.

Note that $m>-na_{ij}\ge -a_{ij}(n-k)$.
If $i\in I_{\bar{0}}$, then $p(i)=0$. If $i\in I_{\bar{1}}$, then $a_{ij}$ is even. In both cases, we have
\begin{align*}
&(-1)^{p(kj)p(ri)}q^{(r\alpha_{i},k\alpha_{j})}(-1)^{nrp(i)p(j)+\binom{r}{2}p(i)}((-1)^{p(i)}q_{i})^{-r(na_{ij}+m-1)}\\
&=(-1)^{(n-k)rp(i)p(j)+\binom{r}{2}p(i)}((-1)^{p(i)}q_{i})^{-r((n-k)a_{ij}+m-1)}.
\end{align*}
Hence each summand of $\varrho^{j,k}(F_{i,j,m,n,\mathbf{c}})$ is a scalar multiple of $F_{i,j,m,n-k,\mathbf{c}'}$.

On the other hand, $$\varrho^{i}(a_{i}^{(r)}a_{j,\mathbf{c}}a_{i}^{(s)})=q_{i}^{r-1}a_{i}^{(r-1)}a_{j,\mathbf{c}}a_{i}^{s}+(-1)^{p(i)p(ri)+p(i)p(nj)}q^{(\alpha_{i},r\alpha_{i}+n\alpha_{j})}q^{s-1}a_{i}^{r}a_{j,\mathbf{c}}a_{i}^{(s-1)}.$$
Then
$$\varrho^{i}(F_{i,j,m,n})=(-1)^{np(i)p(j)}q_{i}^{m-1+na_{ij}}(1-(-1)^{p(i)(m-1)}q_{i}^{-2(m-1+na_{ij})})F_{i,j,m-1,n,\mathbf{c}}.$$
By induction and lemma \ref{1.6}, we get $F_{i,j,m,n,\mathbf{c}}=0$.
\end{proof}

\vskip 3mm
\section{The Quantum Borcherds-Bozec Superalgebras}

 \begin{definition} \label{def:BBsuper}
 {\rm
The {\it quantum Borcherds-Bozec superalgebra} $U_{q}(\g)$ associated with
 a super Borcherds-Bozec Cartan datum $(A, P, P^{\vee},
 \Pi,\Pi^{\vee}, I_{\bar{1}})$ is associative $\Q(q)$ superalgebra
 generated by the elements $a_{il}$, $b_{il}$ $((i,l)\in I^{\infty})$ and $q^{h}, h\in P^{\vee}$ where the parity is given by $p(a_{il})=p(b_{il})=p(li)$ and $p(q^{h})=0$
 with defining relations
 \begin{equation} \label{eq:BBsuper}
  \begin{aligned}
  & q^{0}=0,\ \ q^{h}q^{h'}=q^{h+h'}, \\
  & q^{h}a_{il}q^{-h}=q^{l\langle h,\alpha_{i}\rangle}a_{il},\ \ q^{h}b_{il}q^{-h}=q^{-l\langle h,\alpha_{i}\rangle}b_{il}, \\
  & \sum_{t+t'=1-ka_{ij}}(-1)^{t'}(-1)^{p(i)(t'p(kj)+\binom{t'}{2})}\begin{bmatrix}1-ka_{ij}\\t'\end{bmatrix}_{i}a_{i}^{t}a_{jk}a_{i}^{t'}=0,\ \ i\in I^{re}, i\neq(j,k)\in I^{\infty}, \\
  & \sum_{t+t'=1-ka_{ij}}(-1)^{t'}(-1)^{p(i)(t'p(kj)+\binom{t'}{2})}\begin{bmatrix}1-ka_{ij}\\t'\end{bmatrix}_{i}b_{i}^{t}b_{jk}b_{i}^{t'}=0,\ \ i\in I^{re}, i\neq(j,k)\in I^{\infty}, \\
  & a_{il}a_{jk}-(-1)^{p(li)p(kj)}a_{jk}a_{il}=b_{il}b_{jk}-(-1)^{p(li)p(kj)}b_{jk}b_{il}=0,\ \ \text{if}\ a_{ij}=0,  \\
  & a_{il}b_{jk}-(-1)^{p(li)p(kj)}b_{jk}a_{il}=\delta_{ij}\delta_{kl}\frac{K_{i}^{l}-K_{i}^{-l}}{1-(-1)^{p(li)}q_{i}^{2l}},\ K_{i}^{l}=q^{ld_{i}h_{i}}.
  \end{aligned}
 \end{equation}
We shall denote by $U$ (resp. $U^{+},\ U^{0}$ and $U^{-}$) for $U_{q}(\g)$ (resp. $U_{q}^{+}(\g),\ U_{q}^{0}(\g)$ and $U_{q}^{-}(\g)$).

 }
 \end{definition}
There is a unique automorphism $\omega:\ U\to U$ such that
$$\omega(a_{il})=(-1)^{p(li)}b_{il},\ \omega(b_{il})=a_{il},\ \omega(q^{h})=q^{-h}$$
for $(i,l)\in I^{\infty},\ h\in P^{\vee}$. We have $\omega(x^{+})=(-1)^{p(x)}x^{-},\ \omega(x^{-})=x^{+}$ for all $x\in F$.

Similarly, there is a unique isomorphism of $\Q(q)$ vector spaces $\sigma: U\to U$ such that
$$\sigma(a_{il})=a_{il},\ \sigma(b_{il})=(-1)^{p(li)}b_{il},\ \sigma(q^{h})=q^{-h}$$
for $(i,l)\in I^{\infty},\ h\in P^{\vee}$ such that $\sigma(uu')=\sigma(u')\sigma(u)$ for $u,u'\in U$. We have
$\sigma(x^{+})=\sigma(x)^{+},\ \sigma(x^{-})=(-1)^{p(x)}\sigma(x)^{-}.$

 \vskip 2mm
\begin{lemma}
There is a unique algebra homomorphism $\Delta:\ U\to U\otimes U$, where $U\otimes U$ is regarded as a superalgebra in the standard way, given by
$$\Delta(q^{h})=q^{h}\otimes q^{h},$$
$$\Delta(a_{il})=a_{il}\otimes 1+K_{i}^{l}\otimes a_{il},$$
$$\Delta(b_{il})=b_{il}\otimes K_{i}^{-l}+1\otimes b_{il}.$$
\end{lemma}

Define maps $j^{\pm}:\ F\otimes F\to U\otimes U$ by
$$j^{+}(x\otimes y)=x^{+}K_{|y|}\otimes y^{+},\ j^{-}(x\otimes y)=x^{-}\otimes K_{-|x|}y^{-}. $$
These maps are algebra homomorphisms and satisfy
$$j^{+}\varrho(x)=\Delta(x^{+}),\ \ j^{-}\bar{\varrho}(x)=\Delta(x^{-}).$$

Write $\varrho(x)=\sum x_{1}\otimes x_{2}$, we have
$$\Delta(x^{+})=\sum x_{1}^{+}K_{|x_{2}|}\otimes x_{2}^{+} ,$$
$$\Delta(x^{-})=\sum (-1)^{p(x_{1})p(x_{2})}q^{-(|x_{1}|,|x_{2}|)}x_{2}^{-}\otimes K_{-|x_{2}|}x_{1}^{-}.$$
In particular, when $i\in I^{re}$, we have
$$\Delta(a_{i}^{(n)})=\sum_{t+t'=n}q_{i}^{tt'}a_{i}^{(t)}K_{i}^{t'}\otimes a_{i}^{(t')},$$
$$\Delta(b_{i}^{(n)})=\sum_{t+t'=n}((-1)^{p(i)}q_{i})^{-tt'}b_{i}^{(t)}\otimes K_{i}^{-t}b_{i}^{(t')}.$$
\vskip 2mm

\begin{proposition}\label{2.3}
For $x\in F$ and $(i,l)\in I^{\infty}$, we have
\begin{itemize}
\item[(a)] $x^{+}b_{il}-(-1)^{p(li)p(x)}b_{il}x^{+}=\frac{\varrho_{i,l}(x)^{+}K_{i}^{l}-(-1)^{p(li)p(x)-p(li)}K_{i}^{-l}\varrho^{i,l}(x)^{+}}{1-(-1)^{p(li)}q_{i}^{2l}},$
\item[(b)]$a_{il}x^{-}-(-1)^{p(li)p(x)}x^{-}a_{il}=\frac{K_{i}^{l}\varrho^{i,l}(x)^{-}-(-1)^{p(li)p(x)-p(li)}\varrho_{i,l}(x)^{-}K_{i}^{-l}}{1-(-1)^{p(li)}q_{i}^{2l}}.$
\end{itemize}

\end{proposition}
\begin{proof}
(a) holds for the generators. Assume that (a) is correct for $x'$ and $x''$, we will show it holds for $x=x'x''$.
\begin{align*}
x'^{+}x''^{+}b_{il}&=x'^{+}((-1)^{p(li)p(x'')}b_{il}x''^{+}+\frac{\varrho_{i,l}(x'')^{+}K_{i}^{l}-(-1)^{p(li)p(x'')-p(li)}K_{i}^{-l}\varrho^{i,l}(x'')^{+}}{1-(-1)^{p(li)}q_{i}^{2l}})\\
&=(-1)^{p(li)p(x'')}\frac{\varrho_{i,l}(x')^{+}K_{i}^{l}-(-1)^{p(li)p(x')-p(li)}K_{i}^{-l}\varrho^{i,l}(x')^{+}}{1-(-1)^{p(li)}q_{i}^{2l}}x''^{+}\\
&+x'^{+}\frac{\varrho_{i,l}(x'')^{+}K_{i}^{l}-(-1)^{p(li)p(x'')-p(li)}K_{i}^{-l}\varrho^{i,l}(x'')^{+}}{1-(-1)^{p(li)}q_{i}^{2l}}\\
&+(-1)^{p(li)p(x'x'')}b_{il}x'^{+}x''^{+}\\
&=(-1)^{p(li)p(x'x'')}b_{il}x'^{+}x''^{+}+\frac{x'\varrho_{i,l}(x'')^{+}K_{i}^{l}+(-1)^{p(li)p(x'')}q^{(l\alpha_{i},|x''|)}\varrho_{i,l}(x')^{+}x''^{+}K_{i}^{l}}{1-(-1)^{p(li)}q_{i}^{2l}})\\
&-\frac{(-1)^{p(li)p(x'x'')-p(li)}(K_{i}^{-l}\varrho^{i,l}(x')^{+}x''+K_{i}^{-l}(-1)^{p(li)p(x')-p(li)}q^{(l\alpha_{i},|x'|)}\varrho^{i,l}(x'')^{+})}{1-(-1)^{p(li)}q_{i}^{2l}}\\
&=(-1)^{p(li)p(x'x'')}b_{il}x'^{+}x''^{+}+\frac{\varrho_{i,l}(x'x'')^{+}K_{i}^{l}-(-1)^{p(li)p(x'x'')-p(li)}K_{i}^{-l}\varrho^{il}(x'x'')^{+}}{1-(-1)^{p(li)}q_{i}^{2l}}.\\
\end{align*}
 Apply $\omega^{-1}$ to (a) and  we get (b).
\end{proof}

\vskip 3mm
Let $M$ be a $\U$-module. $M=M_{\bar{0}}\oplus M_{\bar{1}}$, we say that M has a weight space decomposition if
\[M=\bigoplus_{\mu\in P}M_{\mu}, \ M_{\mu}=\{m\in M|q^{h}m=q^{\langle h,\mu\rangle}m,\ \forall h\in P^{\vee}\},\]
such that $M_{\mu}=M_{\mu,\bar{0}}\oplus M_{\mu,\bar{1}}$, where $M_{\mu,\bar{0}}=M_{\mu}\cap M_{\bar{0}} $ and $M_{\mu,\bar{1}}=M_{\mu}\cap M_{\bar{1}} $.

We denote $wt(M)=\{\mu\in \h^{*}|M_{\mu}\neq 0\}.$ When $\text{dim}\ M_{\mu}<\infty$ for all $\mu\in P$, we define the character of $M$
\[\text{ch}\ M=\sum_{\mu\in P}(\text{dim}\ M_{\mu})e^{\mu}\]
where $e^{\mu}$ denotes the additive basis of $\C[\h^{*}]$.
 \vskip 3mm
\begin{definition}
A $\U$ module V is called a highest weight module with highest weight $\lambda$ if there is a non-zero vector $v_{\lambda}$ in V such that
\begin{itemize}
\item[(a)] $V=\U v_{\lambda},$
\item[(b)] $a_{il}v_{\lambda}=0,\ \forall (i,l)\in I^{\infty},$
\item[(c)]$q^{h}v_{\lambda}=q^{\langle h,\mu\rangle}v_{\lambda},\ \forall h\in P^{\vee}.$
\end{itemize}
The vector $v_{\lambda}$ is called the highest weight vector with highest weight $\lambda$ and $V$ has a weight decomposition $V=\oplus_{\mu \le \lambda}V_{\mu}$.
\end{definition}
\vskip 3mm

For $\lambda\in P$, let $R(\lambda)$ be the left ideal of $\U$ generated by $q^{h}-q^{\langle h,\lambda\rangle}1, \ h\in P^{\vee}$ and $a_{il},\ (i,l)\in I^{\infty}$. Let $M(\lambda)=\U\slash R(\lambda)$. Clearly $M(\lambda)$ is a free $U^{-}$-module of rank $1$ and a highest weight module with highest weight $\lambda$ .

Let $V=\U v_{\lambda}$ be an arbitrary  highest weight module with highest weight $\lambda$, we have a surjective $\U$-module homomorphism from $V$ to $M(\lambda)$ by sending $v_{\lambda}$ to $\mathbf{1}$, hence every highest weight module with highest weight $\lambda$ is a quotient module of $M(\lambda)$.

Let $J(\lambda)$ be the sum of all submodules of $M(\lambda)$, it is routine to show that the maximal submodule of $M(\lambda)$ is unique and  exactly equals $J(\lambda)$. Let $V(\lambda)=M(\lambda)\slash J(\lambda)$, we get a irreducible highest weight module with highest weight $\lambda$.

 \vskip 3mm
 \begin{proposition}
 Let $\lambda\in P^{+}$ and let $V(\lambda)=\U v_{\lambda}$ be the irreducible highest weight module with highest weight $\lambda$ and highest weight vector $v_{\lambda}$, $\mu\in \text{wt}(V(\lambda))$, we have
 \begin{itemize}
 \item[(a)] If $i\in I^{re}$,
$b_{i}^{\langle h_{i},\lambda\rangle+1}v_{\lambda}=0$.
 \item[(b)] If $i\in I^{im}$ and $\langle h_{i},\lambda\rangle=0$, then $b_{ik}v_{\lambda}=0$ for all $k>0$.
 \item[(c)] If $i\in I^{im}$, then  $\langle h_{i},\mu\rangle\ge 0$.
 \item[(d)]If $i\in I^{im}$ and $\langle h_{i},\mu\rangle=0$, then $V(\lambda)_{\mu-l\alpha_{i}}=0$ for all $l>0$. In particularly, $b_{il}V(\lambda)_{\mu}=0$.
 \item[(e)]If $i\in I^{im}$ and $\langle h_{i},\mu\rangle\le -la_{ii}$, then $a_{il}V(\lambda)_{\mu}=0.$

 \end{itemize}
 \end{proposition}
 \begin{proof}
Assume that $b_{i}^{\langle h_{i},\lambda\rangle+1}v_{\lambda}\neq 0$
  By proposition \ref{2.3}(b), we have
  $$a_{i}b_{i}^{\langle h_{i},\lambda\rangle+1}v_{\lambda}=\frac{K_{i}^{}\varrho^{i}(a_{i}^{\langle h_{i},\lambda\rangle+1})^{-}-(-1)^{p(i)\langle h_{i},\lambda\rangle}\varrho_{i}(a_{i}^{\langle h_{i},\lambda\rangle+1})^{-}K_{i}^{-1}}{1-(-1)^{p(i)}q_{i}^{2}}v_{\lambda}.$$

 If $p(i)=0$, then $(-1)^{p(i)\langle h_{i},\lambda\rangle}=1$. Otherwise, by the definition of $P^{+}$, $\langle h_{i},\lambda \rangle\in 2\Z$, hence $(-1)^{p(i)\langle h_{i},\lambda\rangle}=1$.

  By lemma \ref{1.7} (a), we have $$\varrho_{i}(a^{\langle h_{i},\lambda \rangle+1})=\varrho^{i}(a^{\langle h_{i},\lambda \rangle+1})=q_{i}^{\langle h_{i},\lambda \rangle}[n]_{i}a_{i}^{\langle h_{i},\lambda \rangle}.$$

Sicne
$$K_{i}b_{i}^{\langle h_{i},\lambda \rangle}v_{\lambda}=q_{i}^{-\langle h_{i},\lambda \rangle}b_{i}^{\langle h_{i},\lambda \rangle}v_{\lambda}=b_{i}^{\langle h_{i},\lambda \rangle}K_{i}^{-1}v_{\lambda},$$
we know that $a_{i}b_{i}^{\langle h_{i},\lambda\rangle+1}v_{\lambda}=0$.

When $(j,l)\neq i $, we have $a_{jl}b_{i}^{\langle h_{i},\lambda\rangle+1}v_{\lambda}=0$.
Since $V(\lambda)$ is irreducible, we get a contradiction.

(b) For each $(i,l)\in I^{\infty}$, if $(j,k)\neq (i,l)$, then $a_{jk}b_{il}v_{\lambda}=b_{il}a_{jk}v_{\lambda}=0$. If $(j,k)= (i,l)$, then
$$a_{il}b_{il}v_{\lambda}=\frac{K_{i}^{l}-K_{i}^{-l}}{1-(-1)^{p(li)}q_{i}^{2l}}v_{\lambda}.$$
Since $\langle h_{i},\lambda\rangle=0$, we get $K_{i}^{l}v_{\lambda}-K_{i}^{-l}v_{\lambda}=0$. Hence, $a_{il}b_{il}v_{\lambda}=0$  and we get a contradiction.

(c) Since $\mu\in \text{wt}(V(\lambda))$, $\mu=\lambda-\beta$ for some $\beta\in Q{+}$ with the form $\beta=\sum_{t=1}^{r}l_{t}\alpha_{i_{t}}$.
since $a_{ij}\le 0$ for all $j\in I$, we have
$$\langle h_{i},\mu\rangle=\langle h_{i},\lambda\rangle-(\sum_{t=1}^{r}l_{t}a_{ii_{t}})\ge0 .$$

 (d) For any $u\in V(\lambda)_{\mu}$, write $u=b_{i_{1}l_{1}}\cdots b_{i_{r}l_{r}}v_{\lambda}$. If $\langle h_{i},\mu\rangle=0$, then $\langle h_{i},\lambda\rangle=0$ and $a_{i,i_{t}}=0$ for all $1\le t\le r$. From the definition of $\U$ and (b), we have $b_{il}u=(-1)^{p(li)p(u)}b_{i_{1}l_{1}}\cdots b_{i_{r}l_{r}}b_{il}v_{\lambda}=0$. We have proved (d).

 (e) Suppose $a_{il}V(\lambda)_{\mu}\neq 0$, then $\mu+l\alpha_{i}\in \text{wt}(V(\lambda))$. By (c)
 $$0\le \langle h_{i},\mu+l\alpha_{i}\rangle=\langle h_{i},\mu\rangle+la_{ii}\le 0,$$
 which yields  $ \langle h_{i},\mu+l\alpha_{i}\rangle=0$, By (d),  $\mu\notin \text{wt}(V(\lambda))$, we get a contradiction.
 \end{proof}
  \vskip 3mm
 \begin{definition}\label{2.6}
The category $O_{int}$ consists of $\U$-modules $M$ such that
\begin{itemize}
\item[(a)] $M$ has a weight decomposition $M=\oplus_{\mu\in P}M_{\mu}$ and $\text{dim} M_{\mu}<\infty$ for all $\mu\in P$.
\item[(b)]There exists finitely many weights $\lambda_{1},\cdots ,\lambda_{r}\in P$ such that $\text{wt}(M)\subset\cup_{j=1}^{r}(\lambda_{j}-Q_{+})$.
\item[(c)]If $i\in I^{re}$, $b_{i}$ is locally nilpotent on $M$.
\item[(d)] If $i\in I^{im}$, then $\langle h_{i},\mu\rangle\ge0$ for all $\mu\in \text{wt}(M)$.
\item[(e)] If $i\in I^{im}$ and $\langle h_{i},\mu\rangle=0$, then $b_{il}M(\mu)=0$ for all $l>0$.
\item[(f)] If $i\in I^{im}$ and $\langle h_{i},\mu\rangle\le -la_{ii}$, then $a_{il}(M_{\mu})=0$  for all $l>0$.
\end{itemize}
\end{definition}
\begin{remark}
\begin{itemize}
\item[(a)] The irreducible highest weight $\U$-module $V(\lambda)$ with $\lambda\in P^{+}$ is an object of $O_{int}$ .
\item[(b)]  Finite direct sums of modules belonging to $O_{int}$  is also belonging to $O_{int}$,
\end{itemize}
\end{remark}
 \vskip 7mm
\section{Triangular decomposition}
\begin{lemma}
Let $U^{\ge0 }$ (resp. $U^{\le0}$) be the subalgebra of $U$ generated by $U^{0}$ and $U^{+}$ (resp. $U^{-}$ and $U^{0}$), then $U^{\le0}\cong U^{-}\otimes U^{0},\ U^{\ge0}\cong U^{0}\otimes U^{+}.$
\end{lemma}
\begin{proof}
$U^{-}$ is spanned by monomials in $b_{il}$ and we can construct a monomial basis $B^{-}=\{b_{\tau}|\tau\in \Omega\}$ of $U^{-}$ where $\Omega$ is an ordered set. By the defining relations of $U$, we have a surjective homomorphism
$$U^{-}\otimes U^{0}\to U^{\le0}$$
given by
$$b_{\tau}\otimes q^{h}\to b_{\tau}q^{h},$$
where $\tau\in \Omega,\ h\in P^{\vee}$. We will show $b_{\tau}q^{h}\ (\tau\in \Omega,\ h\in P^{\vee})$ are linearly independent.
Note that $B^{-}=\sqcup_{\beta\in Q_{+}}B_{-\beta}$, where $B_{-\beta}=\{b_{\tau}|\text{deg}(b_{\tau})=-\beta\}$. Consider the linear dependence relation $$\sum_{\tau,h}c_{\tau,h}b_{\tau}q^{h}=0\ \ \text{with}\ \tau\in \Omega,\ h\in P^{\vee},\ c_{\tau,h}\in\Q(q),$$
this relation can be written as
$$\sum_{\beta\in Q_{+}}(\sum_{\substack{\text{deg}\ b_{\tau}=-\beta\\h\in P^{\vee}}}c_{\tau,h}b_{\tau}q^{h})=0,$$
which implies
$$\sum_{\substack{\text{deg}\ b_{\tau}=-\beta\\h\in P^{\vee}}}c_{\tau,h}b_{\tau}q^{h}=0,\ \ \forall \beta\in Q_{+}.$$
We write $b_{\tau}=b_{i_{1}l_{1}}\cdots b_{i_{r}l_{r}}$ with $\sum_{i=1}^{r} l_{i}\alpha_{i}=-\beta$ and we have
\begin{align*}
0&=\Delta(\sum_{\substack{\text{deg}\ b_{\tau}=-\beta\\h\in P^{\vee}}}c_{\tau,h}b_{\tau}q^{h})\\
&=\sum_{\substack{\text{deg}\ b_{\tau}=-\beta\\h\in P^{\vee}}}c_{\tau,h}\Delta(b_{\tau})q^{h}\otimes q^{h}\\
&=\sum_{\tau,h}c_{\tau,h}(b_{\tau}q^{h}\otimes q^{-h_{\tau}}q^{h}+(\text{intermediate terms})+q^{h}\otimes b_{\tau}q^{h}).
\end{align*}
Consider the bi-degree $(0,-\beta)$,
\begin{align*}
0&=\sum_{\tau,h}c_{\tau,h}q^{h}\otimes b_{\tau}q^{h}\\
&=\sum_{h}(q^{h}\otimes (\sum_{\tau}c_{\tau,h}b_{\tau}q^{h})).
\end{align*}
Since $q^{h}$ are linearly independent. By the property of tensor product, we have
$$\sum_{\tau}c_{\tau,h}b_{\tau}q^{h}=0,\ \ \forall h\in P^{\vee}.$$
which implies $\sum_{\tau}c_{\tau,h}b_{\tau}=0$. By the independence of $b_{\tau}$, we have $c_{\tau,h}=0$ for all $\tau\in \Omega,\ h\in P^{\vee}$, hence $b_{\tau}q^{h}\ (\tau\in \Omega,\ h\in P^{\vee})$ are linearly independent.
\end{proof}
\begin{theorem}
The quantum Borcherds-Bozec superalgebra $U_{q}(\g)$ has triangular decomposition $U_{q}(\g)\cong U^{-}\otimes U^{0}\otimes U^{+}$.
\end{theorem}
\begin{proof}
By the relations of $U_{q}(\g)$, there is a surjective homomorphism
$$U^{-}\otimes U^{0}\otimes U^{+}\to U_{q}(\g).$$
We now prove the homomorphism is injective. Let $B^{+}=\{a_{\tau}|\tau\in \Omega\}$ denote a monomial basis of $U^{+}$, we need to show that $B=\{b_{\tau}q^{h}a_{\mu}|\tau,\mu\in \Omega,h\in P^{\vee}\}$ is linearly independent.
As in previous lemma, we only have to consider the linear dependence relation
\begin{equation}\label{3.1}
\sum_{\substack{h\in P^{\vee}\\ \text{deg}\ b_{\tau}+\text{deg}\ a_{\mu}=\gamma}}c_{\tau,h,\mu}b_{\tau}q^{h}a_{\mu}=0\end{equation}
for all $\gamma\in Q$. Then we have
\begin{align*}
0=&\Delta(\sum_{\substack{h\in P^{\vee}\\ \text{deg}\ b_{\tau}+\text{deg}\ a_{\mu}=\gamma}}c_{\tau,h,\mu}b_{\tau}q^{h}a_{\mu})\\
=&\sum_{\substack{h\in P^{\vee}\\ \text{deg}\ b_{\tau}+\text{deg}\ a_{\mu}=\gamma}}c_{\tau,h,\mu}(b_{\tau}\otimes q^{-h_{\tau}}+(\text{intermediate terms})+1\otimes b_{\tau})(q^{h}\otimes q^{h})\\
&\times(a_{\mu}\otimes 1+(\text{intermediate terms})+q^{h_{\mu}}\otimes a_{\mu})
\end{align*}
Take a total ordering $\le$ on $Q$ given by the height and lexicographic ordering. Let $\Omega_{0}$ (resp. $\Omega_{1}$) be the set of all $\tau\in \Omega$ (resp. $\mu\in \Omega$) such that $\text{deg}\ b_{\tau}$ (resp. $\text{deg}\ a_{\mu}$) is minimal (resp. maximal)
among the terms appearing in (\ref{3.1}). It is obvious that $\tau\in \Omega_{0}$ if and only if $\mu\in \Omega_{1}$.

Come back to the terms of bi-degree (max,min) in (\ref{3.1}), we have
$$\sum_{\substack{h\in P^{\vee}\\ \tau\in \Omega_{0}\\ \mu\in \Omega_{1}}}(-1)^{p(\mu)p(\tau)}c_{\tau,h,\mu}q^{h}a_{\mu}\otimes b_{\tau}q^{h}=\sum_{\substack{\tau\in\Omega_{0}\\h\in P^{\vee}}}(\sum_{\mu\in\Omega_{1}}(-1)^{p(\mu)p(\tau)}c_{\tau,h,\mu}q^{h}a_{\mu})\otimes b_{\tau}q^{h}=0. $$
Since $b_{\tau}q^{h}\ (\tau\in\Omega_{0},\ h\in P^{\vee})$ are linear independent, we have
$$\sum_{\mu\in\Omega_{1}}(-1)^{p(\mu)p(\tau)}c_{\tau,h,\mu}q^{h}a_{\mu}=0.$$
Therefore $c_{\tau,h,\mu}=0$ for all $\tau\in \Omega_0,\ \mu\in \Omega_1$. Continue this process while maintaining total ordering $\le$ on Q, we can conclude that $c_{\tau,h,\mu}=0$ for all $h\in P^{\vee}\ \tau,\mu\in \Omega$.
\end{proof}
\begin{corollary}
Let $x\in U^{+}$,
\begin{itemize}
\item[(a)] if $xb_{il}=(-1)^{p(li)p(x)}b_{il}x$ for all $(i,l)\in I^{\infty}$, then $x=0,$
\item[(b)] if $a_{il}x^{-}=(-1)^{p(li)p(x)}x^{-}a_{il}$ for all $(i,l)\in I^{\infty}$, then $x=0.$
\end{itemize}

\end{corollary}

\vskip 3mm
For $\nu\in \N I^{\infty}$, write $\nu=\sum_{(i,l)}\nu_{i,l}(i,l)$. We set \begin{align*}
c(\nu)&=\sum_{p<q}(l_{p}\alpha_{p},l_{q}\alpha_{q})\\
e(\nu)&=\sum_{p<q}p(l_{q}\alpha_{q})p(l_{p}\alpha_{p})
\end{align*}
\begin{lemma}
Let $\nu\in \N I^{\infty}$
\begin{itemize}
\item[(a)] There is a unique $\Q(q)$ linear map $S:\ U\to U$ such that
$$S(a_{il})=-K_{i}^{-l}a_{il},\ S(b_{il})=-b_{il}K_{i}^{l},\ S(q^{h})=q^{-h}$$
and $S(xy)=(-1)^{p(x)p(y)}S(y)S(x)$ for all $x,y\in U$.
\item[(b)] For any $x\in U_{\nu}^{+}$, we have
\begin{align*}
S(x^{+})&=(-1)^{\text{ht}\nu}(-1)^{e(\nu)}(-q)^{c(\nu)}K_{-\nu}\sigma(x)^{+}\\
S(x^{-})&=(-1)^{\text{ht}\nu}(-1)^{e(\nu)}q^{-c(\nu)}\sigma(x)^{-}K_{\nu}
\end{align*}
\item[(c)]There is a unique $\Q(q)$ linear map $S':\ U\to U$ such that
$$S'(a_{il})=-a_{il}K_{i}^{-l},\ S'(b_{il})=-K_{i}^{l}b_{il},\ S'(q^{h})=q^{-h}$$
and $S'(xy)=(-1)^{p(x)p(y)}S'(y)S'(x)$ for all $x,y\in U$.
\item[(d)] $SS'=S'S=1$.
\item[(e)] If $x\in U_{\nu}^{+}$, then $S(x^{+})=q^{-\sum l_{i}^{2}(\alpha_{i},\alpha_{i})}S'(x^{+})$ and $S(x^{-})=q^{\sum l_{i}^{2}(\alpha_{i},\alpha_{i})}S'(x^{-})$
 \end{itemize}
\end{lemma}

\vskip 7mm

\section{The quasi-R-Matrix and the Quantum Casimir}
\begin{definition}
Let $U\hat{\otimes}U$ be the completion of $U\otimes U $ with respect to the following sequence ($N\ge 1$)
$$H_{N}=(U^{+}U^{0}\sum_{\text{ht}\ \alpha\ge N} U_{\alpha}^{-})\otimes U+U\otimes (U^{-}U^{0}\sum_{\text{ht}\ \alpha\ge N} U_{\alpha}^{+})$$
where $U_{\alpha}^{+}=\{x\in U^{+}||x|=\alpha\}$ and $U_{\alpha}^{-}=\{x\in U^{-}||x|=-\alpha\}$.
\end{definition}

\begin{proposition}
For any $\alpha\in \N I^{\infty}$. let $B_{\alpha}$ be the basis of $U_{\alpha}^{+}$, $B_{\alpha}^{*}$ be the dual basis with respect to $(\cdot,\cdot)$. There is a unique family of elements $\Theta_{\alpha}\in U_{\alpha}^{-}\otimes U_{\alpha}^{+}$ such that $\Theta_{\alpha}=1\otimes 1$ and $\Theta=\sum_{\alpha}\Theta_{\alpha}\in U\hat{\otimes}U$ satisfies $\Delta(u)\Theta=\Theta\bar{\Delta}(u)$ and $\Theta_{\alpha}=(-1)^{\text{ht}(\alpha)+e(\alpha)}\sum_{b\in  B_{\alpha}}b^{-}\otimes b^{*+}\in U_{\alpha}^{-}\otimes U_{\alpha}^{+}$. We call the element $\theta$ the quasi-$\mathcal{R}$-matrix for U.

\end{proposition}
\begin{proof}
Consider the element $\Theta\in U\hat{\otimes}U$ of the form $\Theta=\sum_{\alpha}\Theta_{\alpha}$ with $\Theta_{\alpha}=\sum_{b,b'\in B_{\alpha}}c_{b',b}b'^{-}\otimes b*^{+},\ c_{b,b'}\in \Q(q)$. The set of $u\in U$ such that $\Delta(u)\Theta=\Theta\bar{\Delta}(u)$ is a subalgebra of U containing $U^{0}$. To prove this set is equal to U, it is necessary and sufficient that it contains $a_{il}$ and $b_{il}$ for all $(il)\in I^{\infty}$. This amounts to showing that
\begin{align*}
&\sum_{b_{1},b_{2}\in B_{\alpha}}c_{b_1,b_2}a_{il}b_{1}^{-}\otimes b_{2}^{*+}+\sum_{b_{3},b_4\in B_{\alpha-l\alpha_{i}}}(-1)^{p(li)p(b_{3})}c_{b_3,b_4}K_{i}^{l}b_{3}^{-}\otimes a_{il}b_{4}^{*+}\\
=&\sum_{b_{1},b_{2}\in B_{\alpha}}(-1)^{p(li)p(b_{2})}c_{b_1,b_2}b_{1}^{-}a_{il}\otimes b_{2}^{*+}+\sum_{b_{3},b_4\in B_{\alpha-l\alpha_{i}}}c_{b_3,b_4}b_{3}^{-}K_{i}^{-l}\otimes b_{4}^{*+}a_{il},
\end{align*}
and
\begin{align*}
&\sum_{b_{1},b_{2}\in B_{\alpha}}(-1)^{p(li)p(b_{1})}c_{b_1,b_2}b_{1}^{-}\otimes b_{il}b_{2}^{*+}+\sum_{b_{3},b_4\in B_{\alpha-l\alpha_{i}}}c_{b_3,b_4}b_{il}b_{3}^{-}\otimes K_{i}^{-l}b_{4}^{*+}\\
=&\sum_{b_{1},b_{2}\in B_{\alpha}}c_{b_1,b_2}b_{1}^{-}\otimes b_{2}^{*+}b_{il}+\sum_{b_{3},b_4\in B_{\alpha-l\alpha_{i}}}(-1)^{p(li)p(b_{4})}c_{b_3,b_4}b_{3}^{-}b_{il}\otimes b_{4}^{*+}K_{i}^{l}.
\end{align*}
Let $z\in U^{+}$, since the inner product is nondegenerate, these equalities are equivalent to
\begin{align*}
&\sum_{b_{1},b_{2}\in B_{\alpha}}c_{b_1,b_2}(a_{il}b_{1}^{-}-(-1)^{p(li)p(b_{2})}b_{1}^{-}a_{il})(b_{2}^{*},z)\\
&+\sum_{b_{3},b_4\in B_{\alpha-l\alpha_{i}}}c_{b_3,b_4}((-1)^{p(li)p(b_{3})}K_{i}^{l}b_{3}^{-}( a_{il}b_{4}^{*},z)-b_{3}^{-}K_{i}^{-l}(b_{4}^{*}a_{il},z))\\
&=0,
\end{align*}
and
\begin{align*}
&\sum_{b_{1},b_{2}\in B_{\alpha}}c_{b_1,b_2}((-1)^{p(li)p(b_{1})}b_{il}b_{2}^{*+}-b_{2}^{*+}b_{il})(b_{1},z)\\
&+\sum_{b_{3},b_4\in B_{\alpha-l\alpha_{i}}}c_{b_3,b_4}(K_{i}^{-l}b_{4}^{*+}(a_{il}b_{3},z)-(-1)^{p(li)p(b_{4})}b_{4}^{*+}K_{i}^{l}(b_{3} a_{il},z))\\
&=0,
\end{align*}
Note that $p(b_{1})=p(b_{2})=p(b_{3})+p(li)=p(b_{4})+p(li)$.  we have
\begin{align*}
&\sum_{b_{1},b_{2}\in B_{\alpha}}c_{b_1,b_2}(1-(-1)^{p(li)}q_{i}^{2l})^{-1}(K_{i}^{l}\varrho^{i,l}(b_{1})^{-}-(-1)^{p(li)p(b_{1})-p(li)}\varrho_{i,l}(b_{1})^{-}K_{i}^{-l})(b_{2}^{*},z)\\
&+\sum_{b_{3},b_4\in B_{\alpha-l\alpha_{i}}}c_{b_3,b_4}(1-(-1)^{p(li)}q_{i}^{2l})^{-1}((-1)^{p(li)p(b_{3})}K_{i}^{l}b_{3}^{-}( b_{4}^{*},\varrho^{i.l}(z))-b_{3}^{-}K_{i}^{-l}(b_{4}^{*},\varrho_{i.l}(z)))\\
&=0,
\end{align*}
and
\begin{align*}
&-\sum_{b_{1},b_{2}\in B_{\alpha}}c_{b_1,b_2}(1-(-1)^{p(li)}q_{i}^{2l})^{-1}(\varrho_{i,l}(b_2)^{+}K_{i}^{l}-(-1)^{p(li)p(b_2)-p(li)}K_{i}^{-l}\varrho^{i,l}(b_2)^{+})(b_{1},z)\\
&+\sum_{b_{3},b_4\in B_{\alpha-l\alpha_{i}}}c_{b_3,b_4}(1-(-1)^{p(li)}q_{i}^{2l})^{-1}(K_{i}^{-l}b_{4}^{*+}(b_{3},\varrho^{i,l}(z))-(-1)^{p(li)p(b_{4})}b_{4}^{*+}K_{i}^{l}(b_{3} ,\varrho_{i,l}(z)))\\
&=0,
\end{align*}
Using the triangular decomposition, this is equivalent to the equalities

\begin{align}
\label{cas1}
\sum_{b_1,b_2}c_{b_1,b_2} (b_2^*,z)\varrho^{i,l}(b_1)
 +\sum_{b_3,b_4}(-1)^{p(li)p(b_4)} c_{b_3,b_4}(b_4^*,\varrho^{i,l}(z))b_3&=0, \\
 \label{cas2}
\sum_{b_1,b_2}c_{b_1,b_2} (-1)^{p(b_1)p(li)-p(li)}(b_2^*,z)\varrho_{i,l}(b_1)
 +\sum_{b_3,b_4} c_{b_3,b_4}(b_4^*,\varrho_{i,l}(z))b_3&=0, \\
 \label{cas3}
\sum_{b_1,b_2}c_{b_1,b_2} (b_1,z)\varrho_{i,l}(b_2)
 +\sum_{b_3,b_4}(-1)^{p(li)p(b_4)}c_{b_3,b_4}(b_3,\varrho_{i,l}(z))b_4^{*}&=0, \\
 \label{cas4}
\sum_{b_1,b_2}(-1)^{p(li)p(b_2)-p(li)}c_{b_1,b_2} (b_1,z)\varrho^{i,l}(b_2)
 +\sum_{b_3,b_4}c_{b_3,b_4}(b_3,\varrho^{i,l}(z))b_4^{*}&=0.
\end{align}

When $c_{b,b'}=(-1)^{e(b)+l(b)}\delta_{b,b'}$, we have
\begin{align*}
\sum_{b}(-1)^{e(\mu)} (b^*,z)\varrho^{i,l}(b)
 -\sum_{b'}(-1)^{e(\mu)} (b'^*,\varrho^{i,l}(z))b'&=0, \\
\sum_{b} (-1)^{e(\mu-(i,l))}(b^*,z)\varrho_{i,l}(b)
 -\sum_{b'} (-1)^{e(\mu-(i,l))}(b'^*,\varrho_{i,l}(z))b'&=0, \\
\sum_{b}(-1)^{e(\mu)} (b,z)\varrho_{i,l}(b)
 -\sum_{b'}(-1)^{e(\mu)} (b',\varrho_{i,l}(z))b'^{*}&=0, \\
\sum_{b} (-1)^{e(\mu-(i,l))}(b,z)\varrho^{i,l}(b)
 -\sum_{b'} (-1)^{e(\mu-(i,l))}(b',\varrho^{i,l}(z))b'^{*}&=0.
\end{align*}
These equalities can be easily verified by checking when $z$ is a basis or dual basis element.

We will show that $\Theta$ is unique. Suppose $\Theta_{\alpha}'$ and $\Theta$ also satisfy the conditions. Then $\Theta-\Theta'=\sum c_{b,b'}b^{-}\otimes b'^{+}$ satisfy \eqref{cas1}-\eqref{cas4} and $c_{b,b}=0$ for $b\in B_{0}$. Suppose $c_{b,b'}=0$ for $\text{ht}\ (\alpha')\textless n$ and $\text{ht}\ \alpha=n$. Then the second sum in \eqref{cas1} is zero, hence $\varrho^{i,l}(\sum_{b_1,b_2}c_{b_1,b_2} (b_2^*,z)b_1)=0$, so $\sum_{b_1,b_2}c_{b_1,b_2} (b_2^*,z)b_1=0$ and $(\sum_{b_1,b_2}c_{b_1,b_2} b_2^*,z)=0$ for all $z\in F\slash \R$. Therefore $c_{b_{1},b_{2}}=0$ for all $b_{1},b_{2}\in B_{\alpha}$.
\end{proof}

\begin{corollary}
We have $\Theta\bar{\Theta}=\bar{\Theta}\Theta=1\otimes 1$ with equality in the completion.
\end{corollary}

$$(a_{il}\otimes 1)\Theta_{\alpha}+(K_{i}^{l}\otimes a_{il})\Theta_{\alpha-li}=\Theta_{\alpha}(a_{il}\otimes 1)+\Theta_{\alpha-li}(K_{i}^{-l}\otimes a_{il}),$$
$$(1\otimes b_{il})\Theta_{\alpha}+(b_{il}\otimes K_{i}^{-l})\Theta_{\alpha-li}=\Theta_{\alpha}(1\otimes b_{il})+\Theta_{\alpha-li}(b_{il}\otimes K_{i}^{l}).$$
Let $\Theta_{\le n}=\sum_{\text{ht}\alpha\le n}\Theta_{\alpha}$, we have
\begin{align*}
&(a_{il}\otimes 1+K_{i}^{l}\otimes a_{il})\Theta_{\le n}-\Theta_{\le n}(a_{il}\otimes 1+K_{i}^{l}\otimes a_{il})\\
&=\sum_{\text{ht}\alpha=n}(K_{i}^{l}\otimes a_{il})\Theta_{\alpha}-\sum_{\text{ht}\alpha=n}\Theta_{\alpha}(K_{i}^{-l}\otimes a_{il}),\\
&(b_{il}\otimes K_{i}^{-l}+1\otimes b_{il})\Theta_{\le n}-\Theta_{\le n}(b_{il}\otimes K_{i}^{l}+1\otimes b_{il})\\
&=\sum_{\text{ht}\alpha=n}(b_{il}\otimes K_{i}^{-l})\Theta_{\alpha}-\sum_{\text{ht}\alpha=n}(b_{il}\otimes K_{i}^{l})\Theta_{\alpha}.
\end{align*}

Let $S$ be the antipode and $\mathbf{m}:\ U\otimes U\to U$ be the multiplication map $x\otimes x'\to xx'$. Applying $\mathbf{m}(S\otimes 1)$ to the identities above, for any $n\ge 0$, we have
\begin{align*}
&\sum_{\text{ht}\alpha\le n}\sum_{b\in B_{\alpha}}(-1)^{\text{ht}\alpha+e(\alpha)}(S(a_{il}b^{-})b^{*+}+(-1)^{p(\alpha)p(li)}S(K_{i}^{l}b^{-})a_{il}b^{*+}\\
&-(-1)^{p(\alpha)p(li)}S(b^{-}a_{il})b^{*+}-S(b^{-}K_{i}^{l})b^{*+}a_{il})\\
&=\sum_{\text{ht}\alpha=n}\sum_{b\in B_{\alpha}}(-1)^{n+e(\alpha)}((-1)^{p(li)p(\alpha)}S(K_{i}^{l}b^{-})a_{il}b^{*+}-S(b^{-}K_{i}^{-l})b^{*+}a_{il}),
\end{align*}
and
\begin{align*}
&\sum_{\text{ht}\alpha\le n}\sum_{b\in B_{\alpha}}(-1)^{\text{ht}\alpha+e(\alpha)}((-1)^{p(\alpha)p(li)}S(b^{-})b_{il}b^{*+}+S(b_{il}b^{-})K_{i}^{-l}b^{*+}\\
&-S(b^{-})b^{*+}b_{il}-(-1)^{p(\alpha)p(li)}S(b^{-}b_{il})b^{*+}K_{i}^{l})\\
&=\sum_{\text{ht}\alpha=n}\sum_{b\in B_{\alpha}}(-1)^{n+e(\alpha)}(S(b_{il}b^{-})K_{i}^{-l}b^{*+}-(-1)^{p(li)p(\alpha)}S(b^{-}b_{il})b^{*+}K_{i}^{l}).
\end{align*}
Let $\Omega_{\le n}=\sum_{\text{ht}\le n}\sum_{b\in B_{\alpha}}(-1)^{\text{ht} \alpha+e(\alpha)}S(b^{-})b^{*+}$. By the definition of $S$, we obtain that
\begin{align*}
&S(a_{il}b^{-})b^{*+}+(-1)^{p(\alpha)p(li)}S(K_{i}^{l}b^{-})a_{il}b^{*+}\\
&=-(-1)^{p(li)p(\alpha)}S(b^{-})K_{i}^{-l}a_{il}b^{*+}+(-1)^{p(li)p(\alpha)}S(b^{-})K_{i}^{-l}a_{il}b^{*+}\\
&=0
\end{align*}
We have
\begin{align*}
&K_{i}^{-l}a_{il}\Omega_{\le n}-K_{i}^{l}\Omega_{\le n}a_{il}\\
&=\sum_{\text{ht}\alpha=n}\sum_{b\in B_{\alpha}}(-1)^{n+e(\alpha)}((-1)^{p(li)p(\alpha)}S(K_{i}^{l}b^{-})a_{il}b^{*+}-S(b^{-}K_{i}^{-l})b^{*+}a_{il}),
\end{align*}
and
\begin{align*}
&\Omega_{\le n}b_{il}-b_{il}K_{i}^{l}\Omega_{\le n}K_{i}^{l}\\
&=\sum_{\text{ht}\alpha=n}\sum_{b\in B_{\alpha}}(-1)^{n+e(\alpha)}(S(b_{il}b^{-})K_{i}^{-l}b^{*+}-(-1)^{p(li)p(\alpha)}S(b^{-}b_{il})b^{*+}K_{i}^{l})
\end{align*}
\begin{proposition}
Let $M\in O_{int}$, then for any $m\in M$, we have $\Omega(m)=\Omega_{\le n}m$ is independent of $n$ for large enough $n$, write $\Omega(m)=\sum_{b}(-1)^{\text{ht}|b|+e(|b|)}S(b^{-})b^{*+}m$. Then we have
$$K_{i}^{-l}a_{il}\Omega=K_{i}^{l}\Omega a_{il},\ \ \Omega b_{il}=b_{il}K_{i}^{l}\Omega K_{i}^{l},\ \ \Omega K_{i}^{l}=K_{i}^{l}\Omega.$$
\end{proposition}
\vskip 7mm
Let $c$ be the linear map defined on the highest weight module $M(\lambda)\in O_{int}$ with highest weight $\lambda$, such that $$c(m)=q^{f(\mu)}\Omega m,\ \text{if} \ m\in M_{\mu},$$
where $f(\mu)=(\mu,\mu+2\rho)$ and $\rho$ is defined by $(\alpha_{i},2\rho)=(\alpha_{i},\alpha_{i})$ for any $i\in I$. Know that $\forall (i,l)\in I^{\infty}$,
$$f(\mu-l\alpha_{i})-f(\mu)+2(l\alpha_{i},\mu)=(l^{2}-l)(\alpha_{i},\alpha_{i}).$$
Since $\Omega b_{il}=b_{il}\Omega K_{i}^{2l}$, for any $m\in M_{\mu}$, we have
\begin{align*}
c(b_{il}m)&=q^{f(\mu-l\alpha_{i})}\Omega b_{il}m\\
&=q^{f(\mu-l\alpha_{i})}b_{il}\Omega K_{i}^{2l}m\\
&=q^{f(\mu-l\alpha_{i})+2(l\alpha_{i},\mu)}b_{il}\Omega m\\
&=q^{f(\mu-l\alpha_{i})+2(l\alpha_{i},\mu)-f_{\mu}}b_{il}c(m)\\
&=\begin{cases}
 q^{l(l-1)(\alpha_{i},\alpha_{i})}b_{il}c(m) \ \ & \ \text{if} \ i \in I^{\text{im}}, \\
 b_{i}c(m) \ \ & \ \text{if} \ i \in I^{\text{re}}.
 \end{cases} \\
\end{align*}
Setting $\lambda-\mu=\sum(i_{k},l_{k})\in \N I^{\infty}$, we have
$$c(m)=q^{f(\mu)}\Omega m=q^{f(\lambda)+\sum l_{k}(l_{k}-1)(\alpha_{k},\alpha_{k})}m.$$
If $M(\lambda)$ has a nontrivial submodule $M_{1}$, there is a nonzero element $v_{\mu}=\prod b_{i_{k}l_{k}}v_{\lambda}\neq 0$ such that $a_{il}v_{\mu}=0$, for any $(i,l)\in I^{\infty}$. Moreover  $\Omega(v_{\mu})=1$, hence
$$f(\mu)-f(\lambda)=\sum l_{k}(l_{k}-1)(\alpha_{k},\alpha_{k}),$$
and
$$2(\lambda,\sum l_{k}\alpha_{k})=2\sum_{p<q}l_{p}l_{q}(\alpha_{p},\alpha_{q})$$
If $\lambda\in P^{+}$, $(\lambda,l_{k}\alpha_{k})=0$. By the definition of $O_{int}$, $b_{i_{k}l_{k}}v_{\lambda}=0$, So $M_{1}=0$ and we get a contradiction. We have the following proposition.
\begin{proposition}\label{4.5}
If $M_{\lambda}\in O_{\text{int}}$ is a highest weight module with highest weight $\lambda\in P^{+}$, then $M_{\lambda}$ is irreducible.
\end{proposition}

\begin{proposition}
Let $V=\U v_{\lambda}\in O_{int} $ be a highest weight module with highest weight $\lambda\in P$, then $\lambda\in P^{+}$.
\end{proposition}
\begin{proof}
If $i\in I^{im}$ , by the definition of $O_{int}$, we have $\langle h_{i},\lambda\rangle\ge0$.

If $i\in I^{re}$, since $b_{i}$ is locally nilpotent, we can find an integer $r\ge1$, such that $b_{i}^{(r)}v_{\lambda}=0$ and $b_{i}^{(r-1)}v_{\lambda}\neq 0$. So
\begin{align*}
0=a_{i}b_{i}^{(r)}v_{\lambda}=\frac{q_{i}^{-(r-1)+\langle h_{i},\lambda\rangle}-(-1)^{p(i)(r-1)}q_{i}^{(r-1)-\langle h_{i},\lambda\rangle}}{1-(-1)^{p(i)}q_{i}^{2}}b_{i}^{(r-1)}v_{\lambda}.
\end{align*}
Hence $r-1=\langle h_{i},\lambda\rangle\ge 0$ and if $p(i)=1$, $r-1$ and $\langle h_{i},\lambda\rangle$ are even.
\end{proof}
\vskip 3mm
\begin{remark}\label{4.7}
Every simple module in $O_{int}$ is a highest weight module with highest weight $\lambda\in P^{+}$.
\end{remark}
\vskip 3mm

\begin{lemma}
Let $M$ be a $\U$-module in the category $O_{int}$ and let $V=\U v_{\lambda}$ be a submodule of $M$ with highest weight vector $v_{\lambda}$ of highest weight $\lambda$. Then we have
 $$M\cong V\oplus M\slash V.$$
\end{lemma}
\begin{proof}
We have a short exact sequence
\[0 \to V \xrightarrow{\iota} M \to M/V \to 0,\]
where $\iota$ is an embeding.
According to the lemma $5.9$ in \cite{KK20}, we have a $\U$-module homomorphism $t: M\to V$, such that $t \circ\iota=id_{V}$
which implies the short exact sequence is split, hence $M\cong V\oplus M\slash V.$
\end{proof}

\vskip 3mm
\begin{proposition}
Any integrable module in $O_{int}$ is a direct sum of irreducible highest weight modules.
\end{proposition}
\begin{proof}
For any $M$ in the category $O_{int}$, if $M=\U V$ and  $V$ is a finite dimensional $U^{\ge0}$-submodule. By induction on the dimension of $V$, we can conclude that $M$ is a direct sum of irreducible submodules. So, for each $m\in M$,  $\U m$ is completely irreducible. Hence $M=\sum_{m\in M}\U m$ can be written in the sum of irreducible  $\U$-submodules. From Proposition 3.12 in \cite{CR} and Remark \ref{4.7}, M is a direct sum of irreducible highest weight modules with highest weight $\lambda\in P^{+}$.
\end{proof}

\vskip 3mm
We define $E_{\lambda}$ to be the set of elements of the form
 $\alpha = \sum_{k=1}^{r} a_{k} \alpha_{i_k}$ $(a_{k} \in \Z_{>0})$
 satisfying the following conditions:

 \begin{itemize}

\item[(i)] all $\alpha_{i_k}$ are {\it even} imaginary simple roots for $1 \le k \le r$,

\vskip 2mm

\item[(ii)] $(\alpha_{i_k}, \alpha_{i_l}) = 0$ for all $1 \le k, l  \le r$,

\vskip 2mm

\item[(iii)] $(\alpha_{i_k}, \lambda) = 0$ for all $1 \le k \le r$.

 \end{itemize}

 \vskip 2mm

 For an element $\alpha = \sum_{k=1}^{r} a_{k} \alpha_{i_k}  \in E_{\lambda}$, we define
 \begin{equation} \label{eq:signE}
 \begin{aligned}
 & d_{i}(\alpha) = \begin{cases}
 \#\{k \mid i_k = i \} \ \ & \ \text{if} \ i \notin I^{\text{iso}}, \\
 \sum_{i_k = i} a_{k} \ \ & \ \text{if} \ i \in I^{\text{iso}},
 \end{cases} \\
 & \epsilon(\alpha) =
 \prod_{i \notin I^{\text{iso}}} (-1)^{d_{i}(\alpha)} \prod_{i \in I^{\text{iso}}} \phi(d_i(\alpha)),
 \end{aligned}
 \end{equation}
 where $\phi(n)$ is defined by $\prod_{k=1}^{\infty} (1 - q^k) = \sum_{n=0}^{\infty} \phi(n) q^n$.

 \vskip 2mm

 On the other hand, we define $O_{\lambda}$ to be the set of elements of the form
 $\beta = \sum_{l=1}^{s} b_{l} \alpha_{i_l}$ $(b_{l} \in \Z_{>0})$
  satisfying the following conditions:
 \begin{itemize}

\item[(i)] all $\alpha_{i_l}$ are {\it odd} imaginary simple roots for $1 \le l \le s$,

\vskip 2mm

\item[(ii)] $(\alpha_{i_k}, \alpha_{i_l}) = 0$ for all $1 \le k, l  \le s$,

\vskip 2mm

\item[(iii)] $(\alpha_{i_l}, \lambda) = 0$ for all $1 \le l \le s$.

 \end{itemize}

 \vskip 2mm

 For an element $\beta =  \sum_{l=1}^{s} b_{l} \alpha_{i_l}  \in O_{\lambda}$, we define
 \begin{equation} \label{eq:signO}
 \begin{aligned}
 & d_{i}(\beta) = \begin{cases}
 \#\{l \mid i_l = i \} \ \ & \ \text{if} \ i \notin I^{\text{iso}}, \\
 \sum_{i_l = i} b_{l} \ \ & \ \text{if} \ i \in I^{\text{iso}},
 \end{cases} \\
 & \epsilon(\beta) =
 \prod_{i \notin I^{\text{iso}}} (-1)^{d_{i}(\beta)} \prod_{i \in I^{\text{iso}}}
 c(\lambda - d_{i}( \beta) \alpha_{i}).
 \end{aligned}
 \end{equation}

\vskip 2mm

Finally, we define $F_{\lambda}$ to be the elements of the form $s = \alpha + \beta$
with $\alpha \in E_{\lambda}$, $\beta \in O_{\lambda}$ such that
$(\alpha, \beta) = 0$.
Set
$\epsilon(s) = \epsilon(\alpha) \epsilon(\beta)$,
and define
$$S_{\lambda} = \sum_{s \in F_{\lambda}} \epsilon(s) e^{-s}.$$

\vskip 2mm

\begin{theorem} \label{thm:main}
{\rm
Let  $V(\lambda)$ be the irreducible highest weight module with
a highest weight $\lambda \in P^{+}$.
Then the character of $V(\lambda)$ is given by the formula
\begin{equation} \label{eq:main}
(e^{\rho} R) \text{ch} V(\lambda) = \sum_{w \in W} \epsilon(w) e^{w(\lambda + \rho)} w(S_{\lambda}).
\end{equation}
}
\end{theorem}

\end{document}